\documentclass[
]{amsart}

\usepackage[utf8]{inputenc} 
\usepackage[T1]{fontenc}
\usepackage{xcolor}

\usepackage{xpatch}

\usepackage{ulem}

\usepackage{cancel} 
\usepackage{soul}

\usepackage{amsmath}
\usepackage{amsfonts}
\usepackage{amssymb}
\usepackage{amsthm}
\usepackage{setspace}
\usepackage{amsrefs}
\usepackage{esint}

\usepackage{graphicx}

\usepackage{hyperref}
\hypersetup{hidelinks}

\title[]{A homotopy formula for $a_q$ domains in complex manifolds
}
\author[]{Xianghong Gong$^{\dagger}$}

\date{\today}

 \address{Department of Mathematics,
 University of Wisconsin-Madison, Madison, WI 53706}
 \email{gong@math.wisc.edu}

\author[]{Ziming Shi$^{\ddagger}$}

\address{Department of Mathematics,
	University of California-Irvine, Irvine, CA 92697}
\email{zimings3@uci.edu}

\thanks{$^{\dagger}$Partially supported by  NSF grant  DMS-2349865. $^{\ddagger}$Partially supported by AMS-Simons Travel Grant.}
\dedicatory{\hfill Dedicated to the memory of Professor Joseph J. Kohn}



\makeatletter
\newcommand*\bigcdot{\mathpalette\bigcdot@{.75}}
\newcommand*\bigcdot@[2]{\mathbin{\vcenter{\hbox{\scalebox{#2}{$\m@th#1\bullet$}}}}}
\makeatother

\xpatchcmd{\pfn}{.}{\proofpunctuation}{}{}
\xpatchcmd{\pfn}{\itshape}{\prooffont}{}{}
\xpatchcmd{\pfn}{#1}{\proofname\ifx\proofname#1\else #1\fi}{}{}

\newcommand{\proofpunctuation}{:}
\newcommand{\prooffont}{\bfseries}

\xpatchcmd{\pfn}{.}{\proofpunctuation}{}{}

\newtheorem{thm}{Theorem}[section]
\newtheorem{cor}[thm]{Corollary}
\newtheorem{prop}[thm]{Proposition}
\newtheorem{lemma}[thm]{Lemma}

\theoremstyle{definition}

\newtheorem{defn}[thm]{Definition}

\newtheorem{rem}[thm]{Remark}

\renewcommand{\th}[1]{\begin{thm}\label{#1}}
\renewcommand{\eth}{\end{thm}}
\newcommand{\co}[1]{\begin{cor}\label{#1}}
\newcommand{\eco}{\end{cor}}
\renewcommand{\le}[1]{\begin{lemma}\label{#1}}
\newcommand{\ele}{\end{lemma}}
\newcommand{\pr}[1]{\begin{prop}\label{#1}}
\newcommand{\epr}{\end{prop}}

\newcommand{\ga}{\begin{gather}}
\newcommand{\ega}{\end{gather}}
\newcommand{\gan}{\begin{gather*}}
\newcommand{\egan}{\end{gather*}}
\newcommand{\al}{\begin{align}}
\newcommand{\eal}{\end{align}}
\newcommand{\aln}{\begin{align*}}
\newcommand{\ealn}{\end{align*}}
\newcommand{\eq}[1]{\begin{equation}\label{#1}}
\newcommand{\eeq}{\end{equation}}

\newcommand{\ci}{~\cite}

\newcommand{\cc}{{\bf C}}

\newcommand{\rr}{{\bf R}}

\renewcommand{\dbar}{\bar{\partial}}

\newcommand{\cL}{\mathcal}

\newcommand{\re}[1]{(\ref{#1})}
\newcommand{\rea}[1]{$(\ref{#1})$}
\newcommand{\rl}[1]{Lemma~\ref{#1}}

\newcommand{\rp}[1]{Proposition~\ref{#1}}
\newcommand{\rt}[1]{Theorem~\ref{#1}}
\newcommand{\rd}[1]{Definition~\ref{#1}}
\newcommand{\rrem}[1]{Remark~\ref{#1}}

\newcommand{\supp}{\operatorname{supp}}

\newcounter{pp}
\newcommand{\bpp}{\begin{list}{$\hspace{-1em}(\alph{pp})$}{\usecounter{pp}}}
\newcommand{\epp}{\end{list}}

\newcounter{ppp}
\newcommand{\bppp}{\begin{list}{$\hspace{-1em}(\roman{ppp})$}{\usecounter{ppp}}}
\newcommand{\eppp}{\end{list}}

\def\beq{\begin{equation}}
\def\eeq{\end{equation}}

\newcommand{\all}{\alpha}

\newcommand{\del}{\delta}
\newcommand{\Del}{\Delta}
\newcommand{\var}{\varphi}
\newcommand{\e}{\epsilon}

\newcommand{\Om}{\Omega}

\newcommand{\La}{\Lambda}

\newcommand{\ov}{\overline}
\newcommand{\ti}{\tilde}

\newcommand{\yh}{\frac{1}{2}}

\newcommand{\Kc}{\mathcal{K}}

\newcommand{\db}{\dbar}

\newcommand{\bl}{\bigl(}
\newcommand{\br}{\bigr)}
\newcommand{\Bl}{\Bigl(}
\newcommand{\Br}{\Bigr)}

\newcommand{\nn}{\nonumber}

\begin{document}

 \maketitle



\begin{abstract}We construct a global homotopy formula for $a_q$ domains in a complex manifold. The homotopy operators in the formula will  gain $1/2$ derivative in H\"older-Zygmund spaces $\Lambda^{r}$ when the boundaries of the domains are in $\Lambda^{r+3}$ with $r>0$.
\end{abstract}

\setcounter{thm}{0}\setcounter{equation}{0}

\section{Introduction}\label{sect1} 

Let $M=\{\rho<0\}$ be a relatively compact domain in a complex manifold $\cL M$ of dimension $n$, where $\rho$ is a $C^2$ function on $\cL M$.
We say that the domain $M$  satisfies the  {\it condition $a_q$}, if  for each $p$ in the boundary $ bM$,  $d\rho(p)\neq0$ 
and the Levi-form   of $\rho$ on $T^{(1,0)}_p( b M)$ 
  has at least $(n-q)$ positive or $(q+1)$ negative eigenvalues.
Let $V$ be  a holomorphic vector bundle  on $\cL M$. We will denote by
  $\Lambda^r_{q}(M,V)$ the space of    $V$-valued $(0,q)$ forms on $\ov M$ of H\"older-Zygmund class $\Lambda^r$. When $r>0$ is not an integer, $\Lambda^r$ is the standard H\"older space $C^r$; see Section~\ref{sect:preli} for definition when $r$ is a positive integer.  

The main result of this paper is the following.
\th{}\label{thm:decop}Let
$M$ be a relatively compact $a_q$ domain  in a complex manifold  $\cL M$ of dimension $n$. Assume that $1\leq q\leq n-2$, $0<r<\infty$,
and  $b M\in\Lambda^{r_0+3}$ with $r\leq r_0\leq\infty$.
Then there exists a decomposition
\ga{}\label{Hodge-dec}
f=\db P_qf+P_{q+1}\db f+H_qf
\end{gather}
for all $ f\in \Lambda^r_{q}(\ov M, V)$ with $ \db f\in \Lambda^r$.
Furthermore,  we have the following
\bpp  
\item The operators
\al{}
P_{\tilde q}\colon &\Lambda_{\tilde q}^r(M,V)\to  \Lambda_{\tilde q-1}^{r+1/2}(M,V), \quad \tilde q=q,q+1;\\
H_q\colon &\Lambda_{q}^\e(M,V)\to  \cL H_{q}(M,V)\subset\Lambda^{r+1/2}_{q}(M,V),\quad \forall\e>0
\end{align}
are linear and bounded. 
\item $\dim \cL H_{q}(M,V)<\infty$,  $H_q$ is  a projection $($i.e. $H_q^2=I)$, and  each non-zero element $f\in \cL H_{q}(M,V)$ is $\db$-closed and admits no global solution $u\in L^2_{loc}(M,V)$ for $\db u=f$ on $M$.
    \item The $P_{\tilde q}, H_q, \cL H_q(M,V)$ are independent of $f$ and $r$. Consequently, they depend only on $ b M$ and $r_0$;  if $ b M\in C^\infty$ and $f,g$ are in  $C^\infty(\ov M)$, then
$P_{\tilde q}g,H_qf$ are in $ C^\infty(\ov M)$.
\epp
\eth
For a detailed statement including estimates on the norms of $P_q,P_{q+1}$ and $ H_q$, see \rt{defAB+}.  When $H_{\db}^{(0,q)}(M,V)=0$, we  have $\cL H_{(0,q)}(M,V)=0$ by  \ci{MR4866351}*{Thm. 1.1}. Then \re{Hodge-dec} is called a homotopy formula.

\medskip 

We should mention  that condition $a_q$, first appeared in~\ci{MR0179443}, is the same as condition $Z_q$ that plays an important role in the regularity of $\db$-Neumann operator $N_q$ on $(0,q)$ forms on manifolds with smooth boundary (see Kohn-Rossi~\cite{MR0461588} and Folland-Kohn~\cite{MR177135}), which extend earlier work of Kohn on $\db$-Neumann operators on strongly pseudoconvex manifolds with smooth boundary~\cites{MR0153030, MR0208200}. The regularity of $\db$-Neumann operators has an application to the local Newlander-Nirenberg theorem (See \cite{MR0153030}). 
When both $ b M$ and a formal integrable almost complex structure that is a small perturbation of the complex structure on $\cL M$ are smooth ($C^\infty$) and $H^{(0,1)}_{\db}(M,\Theta)=0$ with $\Theta$ being the holomorphic tangent bundle of $\cL M$, Hamilton showed that  the deformed structure can be transformed into the complex structure on $\cL M$ via a $C^\infty$ embedding of $\ov M$ into $\cL M$. Hamilton~\cites{MR0477158, MR594711 } proved this version of global Newlander-Nirenberg theorem by adapting regularity of $\db$-Neumann operators for a family of formally integrable almost complex structures. 

 Our motivation for \rt{thm:decop}  also comes from the study of global Newlander-Nirenberg problem for finite smooth domains and deformed finite smooth complex structures on the domains.
 As mentioned early, when $H^{(0,q)}_{\db}(M, V)=0$, the decomposition \re{Hodge-dec} provides a homotopy formula. In \cites{MR999729, MR995504} Webster introduced the method of homotopy formulas to the study of local Newlander-Nirenberg problem that yields a sharp regularity result and the local CR embedding problem; see Ma-Michel~\ci{MR1263172} and Gong-Webster~\cites{MR2868966} for further results on the local CR embedding. This approach via homotopy formulas was  applied to the study of stability of global CR embedding problem by Polyakov~\cites{MR2088929}.  The homotopy formulas are useful in the study of the global Newlander-Nirenberg problem.  Results in this direction were obtained very recently for strongly pseudoconvex domains in $\cc^n$  by Gan-Gong~\cites{GG} and Shi \cites{MR4853035}.  In~\ci{GS-nn}, the present authors considered the global Newlander-Nirenberg problem for $a_1^+$ domains (i.e. domains satisfying both $a_1$ and $a_2$ conditions). Condition $a_1^+$ is of course stronger than $a_1$,  and this stronger condition allows them to construct a global homotopy formula by using Kohn's regularity result on $\db$-Neumann operators $N_1,N_2$ on $a_1^+$ domains. However, the approach in \cite{GS-nn} does not yield a homotopy formula for $a_1$ or $a_q$ domains. 
 
 As a first step to address the global Newlander-Nirenberg problems on $a_1$ domains beyond Hamilton's work, we construct the homotopy formula for a {\it fixed} domain via \rt{thm:decop}. The proof of the theorem is based on a functional analysis approach. A very general construction of global homotopy formulas through local approximate homotopy formulas and functional analysis   was employed by Leiterer~\ci{MR1453619}. The  functional analysis approach for an analogous problem on embedded compact CR manifolds has been successfully developed by Polyakov~\ci{MR2077422}, and Laurent-Thi\'ebault and Leiterer~\cites{MR1621967, MR2393269}. In this paper, we adapt Polyakov's construction~\ci{MR2077422} for embedded CR manifolds to the case of $a_q$ domains. We believe that  the improvements  in this paper will find its application to the global CR embedding problem studied by Polyakov~\ci{MR2088929}.  
We should mention that there are  different approaches by  Leiterer~\cites{MR1453619},   Laurent-Thi\'ebaut and Leiterer ~\cites{MR1621967,MR2393269}, and Br\"onnle,   Laurent-Thi\'ebaut and Leiterer~\cites{MR2574722} to construct global homotopy formulas for embedded CR manifolds. However, their approaches are different from us. For instance our construction does not rely on the approximate homotopy operators for all small degrees or all large degrees. 
The interested reader can also find recent results on homotopy formulas and more references   in~\cites{MR3961327,MR4866351, MR4244873, MR2829316,MR4289246, MR4688544, SYajm}.
\medskip 

The paper is organized as follows. 

In Section~\ref{sect:glob}, we recall estimates on local approximate homotopy formulas from \ci{GS-nn}. We then use a standard partition of unity procedure to construct a global approximate homotopy formula. 
Our homotopy formula or   decomposition \re{Hodge-dec}  require us to study  compositions of several operators that gain derivatives. To minimize the requirement on the regularity of the boundaries of the $a_q$ domains and facilitate the estimates for the compositions, we will define a simple notion $e_b(T)$ measuring the gain of  derivatives of an  operator $T$ when the boundaries are at least of class $\Lambda^b$; see \rd{e-gain} and \rl{e1e2} for properties of $e_b(T)$.

In Section~\ref{sect3}, we adapt the functional analysis approach by Polyakov~\ci{MR2077422} to our problem. To construct the homotopy formula when $H_{\db}^{(0,q)}(M,V)=0$, we start with an approximate homotopy formula
$$
f-\hat A^{(0)}_qf=\db T_q^{(0)}f+T_{q+1}^{(0)}df
$$
where the regularities of $T_q^{(0)}, T_{q+1}^{(0)}, \hat A_q^{(0)}$ follow from \ci{GS-nn}. 
Then we attempt to invert $A^{(0)}_q:=I-\hat A^{(0)}_q$. This requires us to modify $A^{(0)}_q, T_q^{(0)}, T_{q+1}^{(0)}$  to obtain the new operators $A_q^{(1)}, T_q^{(1)}, T_{q+1}^{(1)}$ and $ T_q^{(2)}, T_{q+1}^{(2)}$ on subspaces of $\Lambda_{\tilde q}^r(M,V)$.
These new operators will be used to construct projection operators onto some finite dimensional spaces arising from the compact operator $\hat A^{(0)}_q$. A preliminary decomposition \re{Hodge-dec} is achieved in  Section~\ref{sect:hf}.

In Section~\ref{sect:improve}, we introduce a procedure to modify Polyakov's construction~\cites{MR2077422} to optimize the regularity gain of the homotopy operators in \rt{thm:decop}.

\setcounter{thm}{0}\setcounter{equation}{0}
\section{Preliminaries} \label{sect:preli}

 In this section, we first  recall elementary properties of H\"older norms $\|\cdot\|_{C^a(M)}$ with $0\leq a <\infty$  and H\"older-Zygmund norms $\|{\cdot}\|_{\Lambda^a(M)}$ with $0< a <\infty$
where $M$ is a Lipschitz   domain.
For abbreviation, we also write $\|\cdot\|_{C^a(M)}$    as $\|\cdot\|_{M,a}$  and   $\|\cdot\|_{\Lambda^a(M)}$  as   $|\cdot|_{M,a}$.

 We will fix a finite open covering $\{U_1,\dots, U_{k_0}\}$ of a neighborhood  of $\overline M$. We may assume $\cL M=\cup U_j$.
For each $k$, we assume
\bppp
\item  there is a holomorphic coordinate map $\psi_k$ from   ${U_k}$ onto  $\hat U_k\Subset \cc^n$, and  $\psi_k$ is biholomorphic on an open set $\tilde U_k$ containing $\ov{U_k}$;
\item
$M_k:=U_k\cap M\neq\emptyset$,
   $\overline{U_k}\subset M$ if $\ov{U_k}\cap b M\neq\emptyset$, and $D_k:=\psi_k(U_k\cap M)$ are Lipschitz domains;
\item there is a partition of unity $\{\chi_k\}$ such that $\supp \chi_k\Subset U_k$, $\chi_k\in C^\infty(\cL M)$, 
     and $\sum\chi_k=1$ on a neighborhood $\cL U$ of $\overline M$, and set $\hat\chi_j=\chi_j\circ\psi_j^{-1}$.
\eppp
We will write $z_k=x_k+iy_k=\psi_k=(z_k^1,\dots, z_k^n)$ as local  holomorphic coordinates on $U_k$. We will write $\nabla^\ell$ as the set of partial derivatives of order $\ell$ in local real coordinates.
When $M'$ is a small $C^2$ perturbation of $M$, the above conditions are still satisfied for $M'$, without changing $\psi_k, U_k, \chi_k, \cL U$.


\medskip

We now use the local charts to define norms.

 Let $r=k+\beta$ with $k\in\nn$ and $0<\beta\leq1$. Let $D$ be a bounded Lipschitz domain in $\rr^d$. Let $\|u\|_{C^r(\ov D)}$  denote the standard H\"older norm for a function $u\in C^r(\ov D)$.   The  H\"older-Zygmund space $\Lambda^\beta( D)$ is the set of functions $ u\in C^0(\ov D)$ such that
$$
\| u\|_{\Lambda^\beta(D)}:=\sup_{x\in D}
| u(x)|+\sup_{x,x+h,x+2h\in D,h\neq0}|h|^{-\beta}|\Delta^2_hu(x)|<\infty.
$$
Here $\Del_hu(x)=u(x+h)-u(x)$ and $\Del^2_hu(x)= u(x+h)+ u(x+2h)-2u(x+h)$. When $r=k+\beta$ with $k\in\nn$ and $0<\beta\leq1$, define
$\|u\|_{\Lambda^r(D)}=\|u\|_{C^k(D)}+\|\nabla^ku\|_{\Lambda^\beta(D)}$.

 For discussions on various equivalent H\"older-Zygmund forms,  see   \ci{MR4866351}*{sect.~5}.

Throughout the paper we assume that $V$ is a holomorphic vector bundle on $\cL M$.
\begin{defn}\label{def-norm-f} When $u$ is a function on $\ov M$, define $u_k=u\circ\psi_k^{-1}$ and
$$
|u|_{M,r}:=\|u\|_{\Lambda^r(M)}:=\max_k|u_k|_{D_k,r}.
$$
Let $f$ be a section $V$-valued $(0,q)$ forms on $\ov M$ where $q=0,1,\dots, n$. Write
$
f|_{U_k}=\sum_{j=1}^m\sum_{|I|=q} \varphi^j_{k,I}(z_k)(d\bar z_k)^I\otimes e_{k,j}(z_k)$.
Set
$$
|f|_{M,r}:=\|f\|_{\Lambda^r(M)}:=\max_{k,I,j}
|f^j_{k,I}|_{D_k,r}.
$$
\end{defn}

 We will use the abbreviation
\ga{}
\Lambda_{q}^s:=\Lambda_{(0,q)}^s(M, V ),\quad \Lambda^+_{q}:=\bigcup_{s>0}\La^s_{(0,q)}(M, V ),\\
 \ti \La^r_q: = \{f\in \La^r_{q} \colon \db f \in \La^r_{q+1}\}, \quad
 \ti \La^+_q: =\bigcup_{r>0}\tilde\Lambda^r_q.
\end{gather}
Note that $\cL Z^s_{q}:=\Lambda^s_{(0,q)}(M, V )\cap \ker \db\subset\tilde\Lambda^s$. We write $\cL Z^+_q=\cup_{s>0}\cL Z^s_q$.

 We have the following facts on norms; see~\cites{MR0602181,MR2829316,MR4866351, GS-nn}.

\pr{zprod}  Let $M$, $\widetilde M$ be  relatively compact Lipschitz  domains in  $\cL M$.
Then for functions $u,v$ and $a,b,c,d\geq0$ we have
\al
\|{uv}\|_{M,a}&\leq C_{a}(\|u\|_{M,a}\|v\|_{M,0}+\|u\|_{M,0}\|v\|_{M,a});
\label{hprule}\\
|{uv}|_{M,a}&\leq C_{a}(|u|_{M,a}\|v\|_{M,0}+\|u\|_{M,0}|v|_{M,a}),\quad a>0;
\label{zprule}\\
\label{habcd-0}
\|u\|_{M,a+b}\|v\|_{\tilde M, c+d}&\leq C_{a,b,c,d}(\|u\|_{ M,a+(b+d)}\|v\|_{\tilde M,c}  + \|u\|_{ M,a}\|v\|_{\tilde M,c+(b+d)});
\\
\label{abcd-0}
|u|_{M,a+b}|v|_{\tilde M, c+d}&\leq C_{a,b,c,d,\e}\|u|_{ M,a+(b+d)-\e/2}\|v\|_{\tilde M,c+\e}\\
&\quad +C_{a,b,c,d,\e}\|u\|_{ M,a+\e}\|v\|_{\tilde M,c+(b+d)-\e/2}, \quad b,d,\e>0.
\nonumber
\end{align}
Here, $C_a,C_{a,b,c,d}, C_{a,b,c,d,\e}$ depend only on the Lipschitz norms of $b M,b\tilde M$.
\epr

\setcounter{thm}{0}\setcounter{equation}{0}
\section{Approximate homotopy formula with compact term} \label{sect:glob}

We will follow the approach of Polyakov~\ci{MR2077422}. The norm estimates rely on \rp{Prop::L0T0_1/2_est}, which is an important ingredient.

Throughout the paper we assume that   $M$ is an $a_q$ domain in $\cL M$.
Let us decompose $bM$ into the disjoint union of $b^+_{n-q}M$ and $b^-_{q+1}M$, where $b^+_{n-q}$ is the union of components of $bM$ on which the Levi-form has at least $(n-q)$ positive eigenvalues. Then $b_{q+1}^-M$  is the union of components of $bM$ on which the Levi-form has at least $(q+1)$ negative eigenvalues.

We need to deal with boundary of $a_q$ domains of which $b_{n-q}^+$ is only $C^2$ while $b_{q+1}^-$ has higher smoothness. Thus we define
$$|\rho|^*_a:=|\rho|^*_{\Lambda^a(\cL N)}:=\max\{\|\rho\|_{C^2(\cL M)},\|\rho\|_{\Lambda^a(\cL N)}\}$$
 for a fixed neighborhood  $\cL N$ of $b_{q+1}^-M$ in $\cL M$. It is easy to see that \re{habcd-0} implies
 \aln{}
 |\rho|^*_{\cL N,a+b}|v|_{\tilde M, c+d}&\leq C_{a,b,c,d,\e}\|\rho\|^*_{ \cL N,a+(b+d)-\e/2}\|v\|_{\tilde M,c+\e}\\
&\quad +C_{a,b,c,d,\e}\|\rho\|^*_{ \cL N,a+\e}\|v\|_{\tilde M,c+(b+d)-\e/2}, \quad b,d,\e>0.
 \end{align*}
To easy notation, we will drop the superscript $\ast$ in $|\rho|^*_{a}$ in some computation.
\begin{prop} \label{conv-est}
Let $r\in(0,\infty)$ and $1\leq q\leq n$. For each $\zeta\in b_{n-q}^+M$, there exist open sets $U,U'$ such that $\zeta\in U'\Subset U$ and for  $f\in\tilde\La^+_{(0,q)}(\ov{M \cap U}) $ we have
\begin{equation}\label{local_hf_cvx}
    f= \db T_q f + T_{q+1} \db f
\end{equation}
on $M \cap U'$, where the homotopy operators $T_q,T_{q+1}$ satisfy
$$ 
\|T_{\tilde q}f\|_{\Lambda^{r+\yh}(\ov {M \cap U'})} \leq C_r(\|\rho\|_2)\|f\|_{\Lambda^r(\ov M\cap U)}.
$$ 
Moreover,  $C_r(\|\rho\|_2)$ is stable under small $C^2$ perturbations of $\rho$,  that is that $C_r(\|\tilde \rho\|_2)$  can be chosen independent of $\tilde\rho$ provided   $\|\tilde\rho-\rho\|_{C^2}<\del$  for some $\del$ depends only on $\rho$.
\end{prop}
\begin{proof}
   The homotopy formula \re{local_hf_cvx} is given by \cite[Theorem 3.5]{MR4866351}. The estimates for $T_q, T_{q+1}$ are given by \cite[Theorem 4.5]{GS-nn}.
\end{proof}

\begin{prop} \label{conc-est}
Let $r\in(0,\infty)$, $\e>0$ and $0<\delta\leq1/2$. Assume that $ 0 
\leq q
\leq n-2$. For each $\zeta\in b_{q+1}^-M$, there exist open sets $U,U'$ such that $\zeta\in U'\Subset U$ and   for $f\in\tilde\La^+_{(0,q)}(\ov{M \cap U}) $ we have
\begin{equation} \label{loc_app_hf_concave}
   f  = \db T_q  f  + T_{q+1} \db  f  + \hat A _q f
\end{equation}
on $M \cap U'$.  
Here we set $T_0f=0$ when $f$ is a function.
The operators $ \hat A_q$ and $T_{\tilde q}$ with $\tilde q=q,q+1$  satisfy  
\aln{}
\|T_{\tilde q}  f \|_{\La^{r+\delta}(\ov{M \cap U'})}&\leq   C_{r,\delta,\e} 
(\|\rho\|_2) \bl \|\rho\|^*_{\La^{r+2+\delta}}\| f \|_{C^\e(\ov{M}\cap U)}+  
\| f \|_{\La^r(\ov M\cap U)} \br,
\\ 
\| \hat A_q  f  \|_{\Lambda^{r}(\ov{M \cap U'})} &\leq C_{r,\e}( \|\rho\|_2)\|\rho\|^*_{\La^{r+2}}  \|  f  \|_{C^\e(\ov M\cap U)}.
\end{align*}
 Moreover,  $C_{r,\delta,\e}(\|\rho\|_2), C_{r,\e}(\|\rho\|_2)$   are  stable under small $C^2$ perturbations of $\rho$.
\end{prop}
\begin{proof}
    The approximate homotopy formula \re{loc_app_hf_concave} is given by \cite[4.7]{MR4866351} where we take $\hat A_q$ to be the $L^{12}_{12;q}$ in that formula. The estimate for $\hat A_q$ follows by the fact that it is a boundary integral. The estimates for $T_q,T_{q+1}$ are given by \cite[Theorem 4.12]{GS-nn}.

     We remark that Laurent-Thi\'ebaut and Leiterer~\cite{MR1621967}*{Prop.~0.7}  proved that there is no {\it local} homotopy formula (i.e with $\hat A_q=0$) near a $(q+1)$ concave boundary for $(0,q)$ forms that has {\it good} estimates such as the ones stated above.
\end{proof}

An importance of   condition $a_q$ is that operator $\hat A^{(0)}_{q'}=I-\db T_{q'}-T_{q'+1}\db$ gains derivatives and hence it is compact in a small neighborhood of a $(q+1)$ concave boundary point for $0\leq q'\leq q$. In a small neighborhood of an $(n-q)$ convex boundary point,  $\hat A_{q'}$ is also compact (in fact vanishing) when $q\leq q'\leq n$. We will explore the two features in different ways in section~\ref{sect:improve}.

If $M$ is $a_q$, we can let $\{U_j, \chi_j\}$ be a partition of unity, with  $\cup U_j\Supset M$, $\sup \chi_j\Subset U'_j$, and $\sum\chi_j=1$ on $\ov M$ such that there exist the local approximate homotopy formulas (given by the above propositions)
$$
f= \db T^{(0)}_{j;q}
f+T^{(0)}_{j;q+1} \db f+\hat A^{(0)}_{j;q}f, \quad f\in\tilde \Lambda_q^+(M\cap U_j).
$$
Define
\begin{gather} \label{T0_exp}
T^{(0)}_qf:=\sum\chi_j T^{(0)}_{j;q}f,\quad
T^{(0)}_{q+1}g:=\sum\chi_j T^{(0)}_{j;q+1}g,
\\ \label{L0_exp}
\hat A^{(0)}_qf:=\sum\chi_j \hat A^{(0)}_{j;q}f- \sum  \db \chi_j\wedge T^{(0)}_{j;q}f.
\end{gather}
Then we have
$$ 
f=\db T^{(0)}_q f+T^{(0)}_{q+1} \db f+\hat A_q^{(0)}f, \quad f\in\tilde\Lambda^+_q(M\cap U_j).
$$ 
Define
\eq{A0f-def}
A_q^{(0)}f:=f-\hat A_q^{(0)}f.
\eeq
To simplify notation, we will drop the subscript $q$ in $A_q^{(0)},\hat A_q^{(0)}$ for most of the time.

Having defined these operators, we now study their properties.  
\begin{prop}\label{Prop::L0T0_1/2_est}
Let $M\Subset \cL M$ be an $a_q$ domain satisfying \rea{ss'}-\rea{ss'-M}. Assume that $1\leq q\leq n-2$   and $0<\del\leq1/2$.
Then we have  for $r>0$
\begin{equation} \label{T0L0_est}
\begin{aligned}
  \|T^{(0)}_{\tilde q} g\|_{\La^{r+\delta}(\ov{M})}&\leq C_{r,\delta,\e} (\rho) \bl |\rho|^*_{\La^{r+2+\delta}}\| g \|_{C^\e(\ov{M})}+  \| g \|_{\La^r(\ov M)} \br,\quad \tilde q=q,q+1;
 \\
 \|\hat A_q^{(0)}f  \|_{\La^{r+\delta}(\ov{M})}&\leq C_{r,\del} (\rho) \bl \|\rho\|^*_{\La^{r+2+\delta}}\| f \|_{C^\e(\ov{M})}+  \| f \|_{\La^r(\ov M)} \br.
\end{aligned}
\end{equation}
We also have
\ga{}\label{A0f}
 A_q^{(0)}f = \db T^{(0)}_q f+T^{(0)}_{q+1} \db f,\quad \text{on $\ti \La^+_q$},\\
\label{T0qt}
T^{(0)}_q(\ti \La^+_q)\subset \ti \La^+_q.
\end{gather}
\bpp
\item  For $s>0$ the space
$\cL Z_q^s$ is  closed in $\Lambda^s_{q}$, and
$$ 
\quad
\db T^{(0)}_q(\ti \La^+_q)\subset\cL Z^+_{q},\quad A^{(0)} (\cL Z^s_{q})\subset\cL Z^s_{q}.
$$ 
\item
Assume further that
\eq{compc}0 <r < r_0+\e,\quad  b M\in\Lambda^{r_0+2+\e},\quad  0<\e\leq 1/2.
\eeq
Then $\hat A_q^{(0)}= I-A_q^{(0)}$ is a compact operator on $\La_{q}^r$, and $A_q^{(0)}$
is a Fredholm operator of index $0$ on $\Lambda^r_q$.
\epp
\end{prop}
\begin{proof}%
One can verify \re{T0L0_est}-\re{T0qt} easily by Propositions~\ref{conv-est} and \ref{conc-est}.

For $f\in \tilde{\Lambda}_q^+$, we have $\db T_q^{(0)}f = A^{(0)}f-T_{q+1}^{(0)} \db f \in\Lambda^+_q$. Then $\db^2=0$ implies that $\db T_q^{(0)}f\in\cL Z^+_q$.

Assume that $0<r<r_0+\e$. Since $\hat A_q^{(0)}$ gains $\min\{\e,r_0+\e-r\}$ derivatives on $\Lambda^r_q$, then it is compact on $\Lambda^r_q$.
\end{proof}

We will need to derive estimates of the form \re{T0L0_est} repeatedly. Notice that the two estimates in \re{T0L0_est} have the same form. Thus, we introduce notion of an operator $T$ gaining   derivatives.
\begin{defn}\label{e-gain}We say that a linear map $T$ gains $\e$ derivative on $\Lambda_q^+$ for $ bM\in \Lambda^{b}\cap C^2$
with $b\geq2$, if the following hold
\bpp\item If for  $0\leq \e'\leq \e$,
$T$ maps $\Lambda^s_{q}( M,V)$ into $ \Lambda^{s+\e'}_{q'}( M,V)$ if  $s>0$ and $ bM\in\Lambda^{b+s+\e'}$.
\item
There exists $\tau\geq0$ such that for all  $ s>0$ and any (small) $ \e''>0$, we have
\al{}
|Tf|_{s+\e'}&\leq C_{\e'',\e',s} (|\rho|^*_{b+s+\e'}|\rho|_{2+\e''}^{*\tau}|f|_{\e''}+|f|_s),\quad \forall\e'\in [0,
 \e].
\label{sube2}
\end{align}
\epp 
When $(a),(b)$ hold, let us write
$$
e_b(T)\geq \e
$$
to indicate the dependence of $\e$ on $b,T$. Let us also
 write $e_b(T)=\infty$, 
 if for all $\e''>0$,  $s>0$ we have
 $$
 |Tf|_{s}\leq C_{\e'',s} |\rho|^*_{b+s}|\rho|_{2+\e''}^{*\tau}|f|_{\e''}.
 $$
\end{defn}
We remark that \rp{Prop::L0T0_1/2_est} requires that both $r$ and $\delta$ be positive. 
Note that with the assumption $s>0$ in the definition,  we have deliberately introduce the $|\rho|^\tau_{2+\e''}$ with $\e''>0$ in \re{sube2} to allow $e_b(T)=0$, i.e. no gain in derivative. This weaker form will also be convenient when we use convexity of H\"older norms (see for instance \re{abcd} below) that is not available to the Zygmund norms, as well as estimates on local approximate homotopy formula in \rp{conc-est} that requires  forms to be in $\Lambda^\e$ and $M\in\Lambda^{2+\e}$ with $\e>0$.

With this weaker form \re{sube2}, \rp{conc-est} implies that
$$ 
e_{2}(T^{(0)}_{\tilde q})\geq 1/2, \quad \tilde q=q,q+1; \qquad e_2(\hat A^{(0)}_q)\geq 1/2.
$$ 
We now prove the following lemma which will be used throughout the paper to deal with a finite composition of operators.
\le{e1e2}Assume that for $i=1,2,$  $T_i\colon \Lambda^+_{q_i}(\Om,V)\to \Lambda^+_{q'_{i}}(\Om,V)$ satisfies $e_{b_i}(T_i)\in[0,1/2]$ with $b_i\geq2$. Then whenever $T_1+T_2$ or $T_1T_2$ are defined, we have \ga{}
e_{\max(b_1,b_2)}(T_1+T_2)\geq \min(e_{b_1}(T_1),e_{b_2}(T_2));\\
\label{et12}
e_{b_1+b_2-2}
(T_2T_1)\geq e_{b_1}(T_1)+e_{b_2}(T_2);\\
e_{b_1+1}(dT_1)\geq\tilde \e_1, \quad \text{when $e_{b_1}(T_1)\geq 1+\tilde \e_1\geq1$}.\label{b11}
\end{gather}
In particular, if $e_{b_1}(T_1)=\infty$, then $e_{b_1+1}(dT)=\infty$.
\ele
\begin{proof}When $b\geq\e/2$ and $d\geq\e/2$ the convexity of H\"older norms implies that
\eq{abcd}
\|u\|_{a+b}\|v\|_{c+d}\leq C_{a,b,c,d}(\|u|_{a+(b+d)-\e/2}\|v\|_{c+\e/2}+\|u\|_{a+\e/2}\|v\|_{c+(b+d)-\e/2})
\eeq
and hence  if $\e>0$ additionally,
\eq{abcd-z}
\|u\|_{a+b}\|v\|_{c+d}\leq C_{a,b,c,d,\e}(|u|_{a+b+d}|v|_{c+\e}+|u|_{a+\e}|v|_{c+b+d}).
\eeq
To verify \re{et12}, we first apply  the definition to get a crude estimate
$$
|T_1f|_{\e''}\leq C_{\e''}(|\rho|^*_{b_1+\e''}|\rho|^{*\tau_1}_{2+\e''}|f|_{\e''}+|f|_{\e''})\leq 2C_{\e''}|\rho|^*_{b_1+\e''}|\rho|^{*\tau_1}_{2+\e''}|f|_{\e''}.
$$
Then we have for $\e_3=\e_1+\e_2$ with $\e_i\leq e(T_i)$
\aln{}
&|T_2T_1f|_{s+\e_1+\e_2}\leq C_{\e',\e_2,s}(|\rho|^*_{b_2+s+\e_1+\e_2}|\rho|_{2+\e''}^{\tau_2}|T_1f|_{\e''}+|T_1f|_{s+\e_1})\\
& \leq C_{\e',\e_2,s}(|\rho|_{b_2+s+\e_1+\e_2}|\rho|_{2+\e''}^{*\tau_2+\tau_1}|\rho|^*_{b_1+\e''} |f|_{\e''}+|\rho|_{2+\e''}^{\tau_1}|\rho|^*_{b_1+s+\e_1}|f|_{\e''}+
|f|_{s}).
\end{align*} 
If $b_1=2$, we get the desired estimate immediately. Suppose $b_1>2$. Recall that $s>0$. Hence, $b_2+s>2$. Let us write $b_2+s+\e_1+\e_2=(2+3\e'')+b_2+s-2+\e_1+\e_2-3\e''$ and $b_1+\e''=(2+3\e'')+(b_1-2-2\e'')$. Note that
$$
|\rho|^*_{b_2+s+\e_1+\e_2}|\rho|^*_{b_1+\e''}\leq C_{b_1,b_2,\e}\|\rho\|^*_{b_2+b_1-2+s+\e_1+\e_2-\e''}\|\rho\|^*_{2+3\e''}.
$$
We have verified \re{et12}. We have \re{b11} as
$$
|dT_1f|_{s+\tilde \e_1}\leq |T_1f|_{s+1+\tilde \e_1}\leq C_{\tilde \e_1,\e_2,s}(|\rho|^*_{b_1+s+1+\tilde \e_1}|\rho|_{2+\e''}^{*\tau_1}|f|_{\e''}+|f|_{s}).
\qedhere
$$
\end{proof}
 
 Formula \re{et12} is especially useful since in our application we will consider composition $T_1\cdots T_k$ with  $e_{b_j}(T_j)$ satisfying $b_j=2$ for all $j$, except at most one.

 Next, we estimate the gain of derivatives for  the inverse operator of $I-\hat A$ when $\hat A$ gains derivatives.
\begin{prop}\label{hatB-est}
Assume that $M$ and $\hat A\colon \Lambda^+_{q}(M,V)\to \Lambda^+_{q}(M,V)$  satisfy
$$ 
e_b(\hat A)\geq\e>0, \quad bM\in \Lambda^{b+\e},  \quad b\geq2.
$$ 
Let $\e''>0$. Then $A=I-\hat A$ satisfies the following:
\bpp\item $\dim \ker A<\infty$. If $h\in\ker A$, then
\eq{hse}
|h|_{s+a}\leq    C_{s,\e'',b,a}(\|\rho\|_2)|\rho|_{2+\e''}^{*\tau} |\rho|^*_{s+b+a}|h|_{\e''},\quad \forall a>0.
\eeq
In particular, for a    bounded projection $\Pi\colon\Lambda_q^{\e''}\to\ker A$, we have  $e_b(\Pi_{\ker A})=\infty$ with
$$
|\Pi_{\ker A}f|_{s+\e}\leq C_{s,a,b,\e,\e''}(\|\rho\|_2)|\rho|_{2+\e''}^{*\tau} |\rho|^*_{s+b+\e} |\Pi_{\ker A}f|_{\e''}.
$$
\item If $\ker A=\{0\}$, then $A^{-1}=I-\hat B\colon \Lambda^+_{q}(\Om,V)\to \Lambda^+_{q}(\Om,V)$ has $e_b(\hat B)\geq\e$.
More specifically, we have
$$
|\hat B_qg|_{s+\e}\leq C_{s,\e,\e'',\tau}
(\|\rho\|_2) 
(|\rho|^{*\tau}_{2+\e''}|\rho|^*_{s+2+\e}(1+|A^{-1}|_{\e''})|g|_{\e''}+|g|_{s}),\quad \forall s>0,
$$
where $|A^{-1}|_{\e''}=\sup\{|A^{-1}f|_{\e''}\colon |f|_{\e''}=1\}$.
\epp
\end{prop}
\begin{proof} $(a)$ Since $\e_b(\hat A)\geq \e>0$, then $\hat A$ is compact on $ \Lambda_q^{\e}$. Thus $\dim ker A <\infty$.   For $0<a<c<b$,
we have $$
|f|_c\leq C_{b}\del^{-c} |f|_{a}+C_b\del^{b-c}|f|_b, \quad \forall \del\in(0,1).
$$
Hence, the identity $h=-\hat A(h)$ implies
\aln{}
|h|_{s+\e}&\leq C_{s,\e''}  (|\rho|^*_{s+b+\e}|\rho|_{2+\e''}^{*\tau}|h|_{\e''}+|h|_s)\\
&\leq C_{s,\e''}  |\rho|^*_{s+b+\e}|\rho|_{2+\e'}^{*\tau}|h|_{\e''}+C^2_{s,\e''}\del^{-s}
|h|_{\e''}+C^2_{s,\e''}\del^\e |h|_{s+\e}.
\end{align*}
When $\rho\in\Lambda^{b+s+\e}$, we have $|h|_{s+\e}<\infty$. 
This gives us (a) by taking $C^2_{s,\e''}\del^\e\leq 1/2$.

$(b)$  Assume that 
 $g\in\Lambda^{s+\e}$ and $\rho\in\Lambda^{b+s+\e}$. By the open mapping theorem on bounded operators, we have $|Bg|_{s+\e}<\infty$. We have $\hat Bg=-\hat Ag+\hat A(\hat B(g))$ and $|A^{-1}|_{\e''}<\infty$ for $0<\e''<\e$.
 Thus \aln{}
&|\hat B g|_{s+\e}\leq |\hat Ag|_{s+\e}+|\hat A(\hat B g)|_{s+\e}\\
&\leq C_{s,b,\e'',\e}(\|\rho\|_2)(|\rho|_{2+\e''}^{*\tau}|\rho|^*_{s+b+\e}|\hat B g|_{\e''}+|\hat B g|_{s}+ |\hat Ag|_{s+\e}) 
\\
&\leq C'_{s,b,\e'',\e}(\|\rho\|_2)(|\rho|_{2+\e''}^{*\tau}|\rho|^*_{s+b+\e}|\hat B g|_{\e''}+ C_{s,\e}(\del^{-s}|\hat B g|_{\e''}+\del^{\e}|\hat B g|_{s+\e})+|\hat Ag|_{s+\e})\\
&\leq  C''(|\rho|_{2+\e''}^{*\tau}|\rho|^*_{s+b+\e}(1+| B|_{\e''})|g|_{\e''}+\del^{-s}(1+| B|_{\e'})|g|_{\e''}+\del^{\e}|\hat B g|_{s+\e})+|\hat Ag|_{s+\e}),
\end{align*}
where $C''=C''_{s,b,\e'',\e}(\|\rho\|_2)$. Since $|\hat B g|_{s+\e}<\infty$, we get $(b)$ by taking $\del$ such that 
\eq{}
C''_{s,b,\e'',\e}(\|\rho\|_2) \del^\e\leq 1/2.
\qedhere
\eeq
\end{proof}

\setcounter{thm}{0}\setcounter{equation}{0}
\section{Existence of projection operators and their norms} \label{sect3}
 The main purpose of this section is to   introduce several finite dimensional subspaces of $\Lambda^s_q$ and study the norms of projections onto these subspaces. 

 To simplify notation, we will not use subscript $q$ in the operators $A^{(0)}$, $A^{(j)},B^{(j)}$ etc. below. Nevertheless they all act on certain  subspaces of the space of $(0,q)$ forms. 
 We will also use $\tilde\Lambda^+_q,\cL R^s,\cL K,\cL S,\cL P$ below for certain subspaces of $ V $-valued $(0,q)$ forms on $\ov M$. Since most results in this section are based on  functional analysis and apply to general differential complex, we will use $d$ instead of $\db$ for the rest of the paper.
 Of course, we will use properties: $(a)$  $d^2=0$;   $(b)$ $d\colon\Lambda^{r}_q\to\Lambda^{r-1}_{q+1}$ is bounded  for $r>1$.
 We will also need $(c)$ if $f\in\tilde\Lambda^\e_q(M)$, then there exists $f_j\in C_q^1(\ov M)$ such that
$f_j-f,d(f-f_j)$ converges to $0$ in the $\Lambda^{\e''}(M)$ norms for $0<\e''<\e$. Note that the simultaneous approximation is satisfied by the $\db$ operator; see also \rrem{density}.

From now on, we assume that
\ga{}\label{ss'}
M\in \Lambda^{r_0+5/2}, \\
\label{ss'-M}
0<s<r_0+1/2\leq\infty.
\end{gather}

\subsection{A decomposition for the total space $\Lambda_{(0,q)}^s$}

We start with the following.
\begin{lemma}[\cite{MR2077422}*{Lemma 4.2}]\label{lemm:An0}
Let $\cL B$ be a Banach space and let $\hat A\colon \cL B\to\cL B$ be linear and compact. Let $A=I-\hat A$. Then there exists  $n_0>0$ such that
\eq{An0} \ker A^{n_0+1}=\ker A^{n_0}.
\eeq Further, $\cL K:=\ker A^{n_0}$ and ${{\cL R}}:=A^{n_0}(\cL B)$ satisfy the following.
\bpp
 \item ${{\cL R}}$ is closed in $\cL B$, ${{\cL R}} \cap \Kc = \{0\}$, and hence
$\cL B={{\cL R}}\oplus \cL K$.
\item $\ker A\subset\cL K$, and $A$ is an isomorphism of ${{\cL R}}$, i.e. on ${{\cL R}}$,  $A$ is linear, injective and onto.
\epp
\ele
Note that $\cL K,{{\cL R}}$ are independent of $n_0$. We always take $n_0=n_0(A)$ to be the smallest positive integer satisfying \re{An0}.

We now apply \rl{lemm:An0} to our operator $A_q^{(0)}$.   
Let $n_0=n_0(A^{(0)})$ be the smallest positive integer as in \rl{lemm:An0} for $A^{(0)}\colon\Lambda^s_q\to\Lambda_q^s$ with $s>0$.  Set $\tilde A^{(0)}=(A^{(0)})^{n_0}.$
By \rl{e1e2}, we have 
$$
e_2(\tilde A^{(0)}-I
)\geq e_2( A^{(0)}-I)\geq1/2. 
$$
 By \rl{lemm:An0}, we have the following
\le{lemma:KSH} Let $ \cL K:=\ker \tilde A^{(0)}$ and $ \cL R ^ s :=\tilde A^{(0)}(\Lambda_q^s)$. Assume that \rea{ss'}-\rea{ss'-M} hold. Then
\begin{gather*}
\Lambda_q^s=\cL R  ^ s \oplus \cL K ,
\intertext{with} \dim \cL K <\infty,\quad \cL K \subset\Lambda_q^{r_0+1/2}, \quad \ov{\cL R  ^ s }=\cL R  ^ s,
\\
A^{(0)}\cL K \subset\cL K , \quad A^{(0)}(\cL R  ^s)=\cL R  ^ s,\quad \ker A^{(0)}\cap \cL R  ^s=\{0\}.
\end{gather*}
Furthermore,  the following hold.
\bpp
\item The $n_0(A^{(0)}), \cL K $ are independent of $s$ and depend on $M$, $r_0$ and $\tilde\Lambda_q^+$.
\item With $\dim \cL K <\infty$, there exists a decomposition
$$\cL K =\cL P \oplus\cL S \oplus\cL H
$$
where
$\cL P \cap \ker d=\{0\}$ and
\eq{defSH}
\cL S :=\cL K \cap d\Lambda^{r_0+1/2}_{q-1},\quad \Kc \cap\ker d=\cL S \oplus\cL H.
\eeq
\item We have
\ga{}
\label{Nscz++}
A^{(0)}=  dT_q^{(0)}\quad  \text{on $\cL S \oplus\cL H$},\qquad
A^{(0)}(\cL S \oplus\cL H) 
\subset\cL S, 
\\
A^{(0)}(\cL R  ^ s \cap\cL Z_q^ s )= \cL R  ^ s \cap\cL Z_q^ s.   
\label{Nscz+}\end{gather}
\epp
\ele
\begin{proof} Since $\hat A^{(0)}: \La_q^r \to \La_q^{r+\yh}$ for any $0 < r \leq r_0$, it is easy to see that $\Kc_q = \ker
\bl (I-\hat A^{(0)})^{n_0} \br$ lies in the space $\La_q^{r_0+\yh}$. Then $\cL K$ is independent of $s$.
Since $\cL S \oplus\cL H =\cL K \cap\ker d$, we have $\cL P \cap \ker d=\{0\}$.

By $A^{(0)}=dT_q^{(0)}+T_{q+1}^{(0)}d$, it is easy to see that
   $A^{(0)}\cL S \subset\cL S $
 and $A^{(0)}\cL H\subset\cL S $. This verifies \re{Nscz++}.

Since $\cL R ^s$ and $ \cL Z_q^s$ are closed in $\Lambda^s_q$, then $\cL R ^s\cap \cL Z_q^s$ is closed in $\Lambda^s_q$. Then $A^{(0)}$ is injective and Fredholm on $\cL R ^s\cap \cL Z_q^s$. Hence we get \re{Nscz+}.
\end{proof}

By definition, $\cL S$ is the space of elements $f$ in $\Kc$ that admit solutions $u$ for $du=f$ with regularity $u\in\Lambda_{q-1}^{r_0+1/2}$, while $\cL H$ is a complement of $\cL S$ in $\cL K$.

\subsection{Inverting $A^{(1)}$ on   $\cL R  ^ s+\cL S $}

\begin{defn}
Fix $s^-_i\in\Lambda_{(0,q-1)}^{r_0+1/2},s_i:=ds^-_i\in\Lambda_q^{r_0}$, and $p_j,h_\ell\in\Lambda_q^{r_0+1/2}$ satisfying
$$ 
\cL H =\oplus_{\ell=1}^{k_0}\cc h_\ell,\quad
\cL S =\oplus_{i=1}^{k_1} \cc s_i,  \quad \cL P _q=\oplus_{j=0}^{k_2} \cc p_j.
$$ 
The decomposition $\Lambda^s_q=\cL R ^s\oplus \cL P\oplus \cL S\oplus\cL H$ induces bounded projections
$
\Pi^s_{\cL P},  \Pi^s_{\cL S}, \Pi^s_{\cL H}, \Pi^s_{\cL R ^s}$ from $\La_q^s$ onto $\cL P, \cL S,\cL H,\cL R ^s$, respectively.
\end{defn}

 It is helpful to observe that $\cL P,\cL S,\cL H$ are independent of $s$, and hence by the uniqueness of the decomposition   on $\cL R^s$, we have $\Lambda^{s}_q\cap \cL R ^{s'}=\cL R ^{s}$ for $0<s'<s<r_0+1/2$. Also
$$
\Pi^s_{\cL P}=\Pi^{s'}_{\cL P},\quad  \Pi^s_{\cL S}=\Pi^{s'}_{\cL S},\quad \Pi^s_{\cL H}=\Pi^{s'}_{\cL H},\quad\Pi^s_{\cL R^s}=\Pi^{s'}_{\cL R^{s'}},\quad  0<s'<s<r_0+1/2.
$$
Consequently, from now on we drop $s$ in $\Pi^s_{\cL P}, \Pi^s_{\cL S},  \Pi^s_{\cL H}$, and write $\Pi_{\cL R}:=\Pi^s_{\cL R^s}$.

We will also use
\gan{}
\Pi_{\cL S}f=\sum\all_i(f)s_i,\quad\Pi^{-}_{\cL S}f=\sum\all_i(f)s_i^-,\\
H_qf:=\Pi_{\cL H}f=\sum\gamma_j(f)h_j.
\end{gather*}
Thus $\Pi_{\cL S}=d\Pi_{\cL S}^-$. We   define operators $T_q^{(1)}$ and  $A^{(1)}$ and rename $T_{q+1}^{(0)}$ as follows.
\begin{defn} On $\Lambda_{q+1}^+$,
set $T_{q+1}^{(1)}=T_{q+1}^{(0)}$. On $\Lambda_{q}^+$, set
$$ 
T_q^{(1)}=T^{(0)}_q(I-\Pi_{\cL S}-\Pi_{\cL H} )+ \Pi^-_{\cL S}.
$$ 
Then for $j=q,q+1$,  $T_j^{(1)}$ still gains $1/2$ derivative on $\Lambda_j^s$.   On $\Lambda^+_q$, set
$$
  A^{(1)}:=A^{(0)}-A^{(0)}(\Pi_{\cL S}+\Pi_{\cL H})+\Pi_{\cL S}.
$$
\end{defn}

Immediately, we have
 \le{} Assume that \rea{ss'}-\rea{ss'-M} hold. Then
\ga
 A^{(1)}=A^{(0)} \quad \text{on}\ \cL R ^ s \oplus\cL P,\quad A^{(1)}|_{ \cL H}=0,\label{A1uAu}
 \quad
 A^{(1)}|_{\cL S}=I.
\end{gather}
\ele

\le{}
On $\tilde \Lambda_q^+$, $A^{(1)}=dT_q^{(1)}+T^{(1)}_{q+1}d$.  Assume that \rea{ss'}-\rea{ss'-M} hold. Then $A^{(1)}-I$ is compact on $\Lambda_q^s$.
 \ele
 \begin{proof}
 By \re{Nscz++},  $dT_q^{(0)}=A^{(0)}$ on $\cL S\oplus\cL H$.
 Then
    the first identity can be checked easily. Write
    \[
      I - A^{(1)} = \hat A^{(0)} + dT_q^{(0)}(\Pi_{\cL S}+\Pi_{\cL H}) - \Pi_{\cL S}.
    \]
    The compactness of $\hat A^{(1)}$ follows from that of $\hat A^{(0)}$.
 \end{proof}

 \le{}Assume that \rea{ss'}-\rea{ss'-M} hold. There exists a bounded isomorphism
$$
B^{(1)}\colon \cL R ^ s \oplus\cL S \to \cL R ^ s \oplus\cL S
$$
such that
\al{}
 &A^{(1)}B^{(1)}=B^{(1)} A^{(1)}=I \quad \text{on $\cL R ^ s \oplus\cL S $};\quad B^{(1)}=I\quad\text{on $\cL S $};\\
\label{B1est}
& |(B^{(1)}-I) g|_{s+\delta}\leq  C_{s,\e,\e''}(\|\rho\|_{2})
 \Bigl(|g|_{s}+\Bigr.\\
 &\hspace{6ex}\Bigl.
 |\rho|^{*\tau}_{2+\e''} |\rho|^*_{s+2+\e}\bigl (1+\bigl|\bigl(A^{(0)}|_{\cL R^{\e''}}\bigr)^{-1}\bigr|_{\e''}\bigr)
  |g|_{\e''}
\Bigr),
  \quad \forall s>0,\   g\in
{\cL R} ^ s \oplus\cL S.
\nonumber
\end{align}
Here $C_{s,\e,\e''}(|\rho|_{2})$ depends only on $\|\rho\|_{2}$.
\ele
\begin{proof} 
Since $A^{(1)}=A^{(0)}$ on $\cL R^s$, then $A^{(1)}$ is  bounded and injective on  $\cL R^s$. Then $A^{(1)}$  is the identity on $\cL S$. Therefore, $A^{(1)}$ has a bounded inverse $B^{(1)}$ on $\cL R^s\oplus \cL S$. 

The estimate for \re{B1est} is  verbatim from the proof of \rp{hatB-est} (b), using $e_2(A^{(0)})\geq1/2$. We leave the details to the reader.
\end{proof}
 
\subsection{Homotopy operators on the subspace $\cL R\oplus \cL S$}
By $A^{(1)}=dT_q^{(1)}+T_{q+1}^{(1)}d$, we get easily
$$ 
A^{(1)}\cL Z_q^ s \subset \cL Z_q^ s, \quad s>0.
$$ 
\le{}
Assume that \rea{ss'}-\rea{ss'-M} hold. Then
\ga\label{ZsN}
\cL Z_q^ s =(\cL R ^ s \cap \ker d)\oplus\cL S \oplus \cL H,\quad \cL Z_q^ s \ominus \cL H:=(\cL R ^ s \cap \ker d) \oplus\cL S,\\
 \label{A1Z} A^{(1)}(Z_q^ s \ominus \cL H)=\cL Z_q^ s \ominus \cL H,\quad  A^{(1)}B^{(1)}=B^{(1)} A^{(1)}=I_{\cL Z_q^s\ominus \cL H}. 
\end{gather}
\ele
\begin{proof}By definition, $\cL S\oplus\cL H\subset\cL Z_q^s$. Thus $\cL Z^s_{(0,r)}=(\cL R^s\oplus\cL P)\cap Z^{s}_{(0,r)}\oplus\cL H\oplus\cL S$.
Let  $f\in\cL  Z^ s_q\cap (\cL R^s\oplus \cL P)$. Decompose
$$
f=f_{\cL R}+ f_{\cL P}, \quad f_{\cL R}\in\cL R^s, \quad f_{\cL P}\in\cL P.
$$
We want to show $f_{\cL P}=0$.

Recall that $\tilde  A^{(0)}= (A^{(0)})^{n_0}$ 
and $\tilde A^{(0)}|_{\cL K}=0$. Then $A^{(0)}\cL Z^s_q\subset\cL Z^s_q$ implies that 
$$
\tilde  A^{(0)}f_{\cL R}=\tilde  A^{(0)}f\in\cL Z_q^s.
$$
Then $ A^{(0)}(\cL R ^s)\subset\cL R ^s$ and \re{A1uAu} imply
 ${\tilde A^{(0)}}f_{\cL R}\in \cL R ^ s \cap\cL Z_q^s.
$
Hence
$f_{\cL R}\in\cL R ^s\cap\cL Z_q^s$ by \re{Nscz+}. Then
$
f_{\cL P}=f-f_{\cL R} \in \cL Z_q^s$. On the other hand, $\cL P\cap\cL Z_q^s=\{0\}$. Thus $f_{\cL P}=0$.
This shows \re{ZsN}.

Since $A^{(1)}$ is an isomorphism on $\cL R^+\oplus\cL S$ with inverse $B^{(1)}$, we get \re{A1Z} from \re{ZsN}.
\end{proof}

\le{max-sol-lemm}Assume that \rea{ss'}-\rea{ss'-M} hold. Assume further that
\eq{max-sol}
\Lambda^{r_0}_{q}\cap d\Lambda_{q-1}^+=\Lambda^{r_0}_{q}\cap d\Lambda^{r_0+1/2}_{q-1}.
\eeq Then
\ga
\label{dTq}
 dT_q^{(i)}(\tilde{\cL R}^\e\oplus\cL S)\subset \cL Z_q^\e\ominus\cL H, 
  \quad i=0,1.
\end{gather}
\ele
\begin{proof}Take $f\in\tilde\Lambda^+_q\cap(\cL R\oplus S)$. Then $dT_q^{(i)}f=A^{(i)}f-T_{q+1}^{(i)}df\in\Lambda^+$;
 and $d^2=0$ implies that $\tilde f:=dT_q^{(i)}f\in \cL Z_q^+$.
By \re{ZsN}, we have the decomposition
$$
\tilde f =\tilde f_{\cL R}+\tilde f_{\cL S}+\tilde f_{\cL H}\in\cL R\oplus\cL S\oplus\cL H.
$$
We want to show that $\tilde f_{\cL H}=0$. By \re{A1Z},  $A^{(1)}$ is an automorphism on $\cL R\cap\ker d\oplus\cL S$ with inverse $B^{(1)}$.
Then $v:=B^{(1)}\tilde f_{\cL R}\in \cL Z_{(0,r)}^+$ and hence $\tilde f_{\cL R}=A^{(1)}v=dT_q^{(1)}v$. We can write $\tilde f_{\cL S}=dw $ with $w\in\Lambda^{r_0+1/2}$.
Therefore,
$$
d(T_q^{(i)}f-T_q^{(1)}v-w)=\tilde f_{\cL H}.
$$
Since $T_q^{(i)}f-T_q^{(1)}v-w\in\Lambda_{(0,q-1)}^+$, then \re{max-sol} implies that $\tilde f_{\cL H}=du$ for some $u\in\Lambda^{r_0+1/2}_{(0,q-1)}$, i.e. $\tilde f_{\cL H}\in\cL S$.  On the other hand,
$\tilde f_{\cL H}\in \cL H$. Therefore, $\tilde f_{\cL H}=0$.
\end{proof}
\begin{rem}
It is proved in \ci{MR4866351} that \re{max-sol} holds for $a_q$ domains $M$ of $\Lambda^{r_0+5/2}$ boundary in a complex manifold; in fact
a stronger result holds: $\Lambda^{r_0}_{q-1}\cap \db L^2_{loc}=\Lambda^{r_0}_{q-1}\cap \db\Lambda^{r_0+1/2}_{q-1}$. However, the result and its proof in \ci{MR4866351} require that $r_0>1$. The same proof is valid for $0<r_0\leq1$ when we use the homotopy formula for $\dbar$-closed  $(0,q)$ forms of class $\Lambda^{r_0}$ for $r_0>0$.
 For the general situation, we need \re{max-sol} as an assumption. Polyakov~\ci{MR2077422} has similar assumptions.
 \end{rem}

\rl{max-sol-lemm} allows us to define the following $T_q^{(2)}, T_{q+1}^{(2)}$.
\le{}  Assume that \rea{ss'}-\rea{ss'-M} and \rea{max-sol}  hold.
 Then
\gan 
T_q^{(2)}f:= T_q^{(1)}(B^{(1)})^2d T_q^{(1)}f,\\ T_{q+1}^{(2)}df:=B^{(1)}T^{(1)}_{q+1}df
=B^{(1)}T^{(0)}_{q+1}df
\end{gather*}
are well-defined for   $ f\in \tilde{\cL R}^+\oplus \cL S$.
\ele
\begin{proof} \rl{max-sol-lemm} says that $T_{q}^{(2)}$ is well-defined.
  Thus $B^{(1)}T_{q+1}^{(1)}df=B^{(1)}A^{(1)}f-B^{(1)}(dT^{(1)}_{q}f)$ is also well-defined for  $ f\in \tilde{\cL R}^+\oplus \cL S$.
\end{proof}
\begin{lemma} Assume that \rea{ss'}-\rea{ss'-M} and \rea{max-sol}  hold.
We have \eq{BAf}
f=dT_q^{(2)}f+T_{q+1}^{(2)}df
\eeq
for  $f\in \tilde {\cL R} ^+ \oplus\cL S $.
Further, for $0<\e\leq 1/2$ and $f\in\widetilde{ \cL R} ^\e \oplus\cL S$ we have
\ga{}
|T_{q+1}^{(2)}df|_{\e}\leq  C_\e|df|_{\e}, 
\label{D0g}\\
\label{no-gain}
|dT_q^{(2)}f|_\e\leq C'_{\e}\|(f, df)\|_{\e}, 
\\
\label{dTq2} dT_q^{(2)}(\widetilde{\cL R}^\e\oplus\cL S)\subset  \cL Z^\e_q \ominus\cL H, 
\end{gather}
where $C_\e,C_\e'$ depend on $|\rho|_{2+\e}$.
\end{lemma}
\begin{proof}
Assume that  $f\in\widetilde{\cL R}^\e \oplus S$. We know that $dT_q^{(1)}f=A^{(1)}f-T_{q+1}^{(1)}df\in\Lambda^\e_q$.  Then $A^{(1)}f$ and $dT_q^{(1)}f$ are in $\cL R ^\e \oplus S $. By the decomposition
\eq{A1f}
A^{(1)}f=d T_q^{(1)}f+T_{q+1}^{(1)}df\in\cL R ^ \e \oplus\cL S,
\eeq
$T_{q+1}^{(1)}df$ is also in $\cL R ^\e \oplus \cL S$.
We now follow \ci{MR2077422} for our situation.
Applying $B^{(1)}$ to both sides of \re{A1f},
we obtain
\eq{BdT}
f=B^{(1)} T_{q+1}^{(1)}df+ B^{(1)}d T^{(1)}_qf.
\eeq
We need to rewrite the last term. Since $B^{(1)} (\cL Z_q^\e\ominus \cL H)\subset\cL Z_q^\e\ominus \cL H$ and $dT_q^{(1)}f\in \cL Z_q^\e\ominus \cL H$, then
\eq{dBdRh}
 d(B^{(1)})^2d T_q^{(1)}f=0.
\eeq
Then
\al{}\label{B1dR}
B^{(1)} d T_q^{(1)}f&=  A^{(1)}B^{(1)}[B^{(1)}d T_q^{(1)}f] =(d T_q^{(1)}+T_{q+1}^{(1)}d)B^{(1)}[B^{(1)}d T_q^{(1)}f]\\
&=d T_q^{(1)}B^{(1)}[B^{(1)}d T_q^{(1)}f] =d T_q^{(2)}f.
\nonumber
\end{align}
We have verified \re{BAf} via \re{BdT}.

The estimate \re{D0g} follows directly from  well-definedness of $T_{q+1}^{(2)}df=B^{(1)}T_{q+1}^{(0)}df$ and estimates of $T_{q+1}^{(0)}$ and $B^{(1)}$ given by \re{T0L0_est} and \re{B1est}. The estimate \re{no-gain} follows from estimates \re{D0g}
and the identity $T_q^{(2)}f=T_q^{(1)}(B^{(1)})^2(A^{(1)}f-T_{q+1}^{(1)}df)$.
Finally, \re{dTq2} follows from  \re{dTq}.
\end{proof}

\subsection{The projection $\Pi_{d\cL P}$}

We are seeking a decomposition of the form $f=dP_qf+P_{q+1}df+H_qf$ for $f\in \tilde\Lambda^s$.  
Although $\tilde \Lambda^s_q$ is still not closed in $\Lambda^s_q$, we have for $s>0$
$$
\tilde\Lambda^s_q=\widetilde{\cL R} ^s \oplus\cL K, \quad \widetilde{\cL R} ^s :=\cL R^s \cap \tilde\Lambda_q^s.
$$
We now show the following auxiliary result to apply Hahn-Banach theorem for $(0,q+1)$ forms, which is an improvement of  Polyakov~\cite[Lemma 4.4]{MR2077422}.

\begin{prop}\label{defbeta_j}Assume that  \rea{ss'}-\rea{ss'-M} and \rea{max-sol}  hold.
Assume further that
\eq{r0>}
d\cL P\subset\Lambda^{\e},
\eeq
where $\e>0$ is fixed.  Then
$$ 
(d\cL P )\cap\ov{d(\widetilde{\cL R}  ^{\e} \oplus\cL S\oplus\cL H)}=(d\cL P )\cap\ov{d\widetilde{\cL R}  ^{\e} }=0,
$$ 
 where $\ov{d\widetilde{\cL R} ^{\e}}$ is the closure in $\Lambda^\e_{(0,q+1)}$ in $\Lambda^{\e}$ norm.
 There exist  bounded linear functionals $\beta_j^{(\e)}$ on $\Lambda_{(0,q+1)}^{\e}$ such that
\eq{beta_j}
\beta_j^{(\e)}(dp_{j'})=\delta_{jj'}, \quad \beta_j^{(\e)}(d\widetilde{\cL R}  ^\e)=0.
\eeq
\end{prop}
\begin{rem}\label{condP} 
When $ b M\in\Lambda^{3+\e}$, we have $\cL P\subset\Lambda^{1+\e}$. 
\end{rem}
\begin{proof}
Suppose that $\psi_i\in\cL P $ and there is a sequence $\phi_i\in\widetilde{ \mathcal R}^\e$ such that
$$
 \var:=\lim_{i\to\infty} d\psi_i=\lim_{i\to\infty} d\phi_i
$$
where each convergence is in $\Lambda^{\e}$ norm and $\var\in\Lambda^\e_{q+1}$. We need to show that $\var=0$.

 Recall that $\cL P$ has a basis $\{p_1,\dots, p_{k_2}\}$. We have $\psi_j=\sum c_{j,\ell}p_\ell$. Since $\|\psi_j\|_\e$ is bounded and $dp_\ell\in\Lambda^\e$, then $c_j=(c_{j,1},\dots, c_{j,k_2})\in\cc^{k_2}$ is also bounded. By taking a subsequence we may assume that $c_j$ converges. Hence  $\psi_j$ converges to $ \psi\in\cL P$ in the space $\Lambda^{r_0+1/2}$  and $d\psi=\var$.

By \re{BAf}, we have for $f\in\tilde{\cL R}^+\oplus\cL S$
$$ 
f=dT^{(2)}_{q}f+T_{q+1}^{(2)}df.
$$ 
Since $T^{(2)}_{q+1}$ does not lose derivative by \re{D0g}, the sequence $T^{(2)}_{q+1}d\phi_i$   still converges to $\gamma:=T^{(2)}_{q+1}\var\in\Lambda^\e$ in $\Lambda^{\e}$ norm.
On the other hand,  $\phi_i\in\tilde{\cL R}^\e\subset\cL R^\e$,  $dT_q^{(2)}\phi_i\in\cL R^\e\oplus\cL S$ by \re{dTq2}, and hence   $T^{(2)}_{q+1}d\phi_i=\phi_i-dT_q^{(2)}\phi_i$ are in $\cL R^\e\oplus\cL S$. The latter is closed. This shows that $\gamma\in \cL R^\e\oplus\cL S$.

Using  $\phi_i=dT^{(2)}_q\phi_i+T_{q+1}^{(2)}d\phi_i$ again, in the sense of distribution  we obtain $d\phi_i=dT_{q+1}^{(2)}d\phi_i$ and
$$d\psi =\lim d\psi_i=\var=\lim d\phi_i=\lim dT^{(2)}_{q+1}d\phi_i=d\gamma.$$
 Then $d(\gamma-\psi)=0$, in the sense of distribution.  Now $\cL P\cap\ker d=\{0\}$ implies that $\gamma-\psi\in\cL Z_q^{\e}\subset{\cL R} ^{\e}\oplus\cL S$, and hence $\psi\in\cL R ^{\e}\oplus\cL S$. However,  $\psi\in\cL P $. Hence $\psi=0$ and $\var=d\psi=0$.

Using \re{r0>} again,   we can define the linear functional $\beta_j$ and verify the rest of the proposition.\end{proof}

\setcounter{thm}{0}\setcounter{equation}{0}
\section{Analogue of Polyakov's homotopy operators on $\tilde \Lambda^+_q$}
\label{sect:hf}

In this section we will achieve the first version of  Hodge-type decompositions for $a_q$ domains. This version of decomposition is analogous to the  decomposition for compact CR manifold in Polyakov~\cites{MR2077422}. We will achieve this by modifying the operators $T_q^{(0)}, T_{q+1}^{(0)}$ and hence $A^{(0)}$.  We want the range $\mathcal R^s $  of  $dP_q+P_{q+1}d$ to be exact $\Lambda^s_q$ under $H^{(0,q)}_{\db}(M,V)=0$ or $\Lambda^s_q$ module a finite dimensional space $\cL K $ that consists of $d$-closed forms.

Recall that there are bounded functional $\all_j^{(\e)}$ on $\Lambda_q^{\e}$ such that $\Pi_{\cL S}f=\sum\all_j^{(\e)}(f)s_i$.
By Hahn-Banach theorem, we can also find bounded functionals $\gamma_\ell^{(\e)}$ on $\Lambda_q^{\e}$ such that $\Pi_{\cL H}(f)=\sum\gamma_\ell^{(\e)}(f)h_\ell$.

Let $\beta_j^{(\e)}$ be as in \rp{defbeta_j}. We now define  
$$ 
\Pi^-_{\cL P}(g)=\sum_{j=1}^{k_2}\beta_j^{(\e)}(g)p_j,\quad \Pi_{d\cL P}:=d\Pi^-_{\cL P}.
$$ 
We remark that     $ \Pi_{d\cL P}$ is a projection from $\Lambda_{q+1}^\e$ onto $d\cL P$ under condition \re{r0>}.
Although the $\beta_j^{(\e)}$ depends on $\e$, \re{beta_j} ensures that $\Pi_{d\L P}$, when restricted on $d\tilde\Lambda^{\e}_q$, does not depend on $\e$. 
By \re{hse}, we have for $s>0$
\aln{} 
|\Pi_{d\cL P}g|_{s}&\leq C|g|_{\e''}|(dp_1,\dots,d p_{k_3})|_{s}\\
&\leq
 C_{s,\e''}(|\rho\|_2)|\rho|_{s+3}|(p_1,\dots, p_{k_3})|_{\e''}|g|_{\e''}.
 \nonumber
\end{align*}
In particular, we have
\eq{e52dP}
e_2(\Pi_{\cL P}^-)=\infty, \quad e_{3}(\Pi_{d\cL P})=\infty.
\eeq

We will now use the four projections $\Pi_{\cL P}, \Pi_{\cL H}, \Pi_{\cL S}, \Pi_{d\cL P}$ to define homotopy operators on $\tilde\Lambda_q$ and $\Lambda_{q+1}$. This will give us a homotopy formula on $\tilde\Lambda_q$.

Define
\al{}\label{T3_q}
T^{(3)}_qf&:= T_q^{(0)}f+\tilde T_q^{(3)}f,\quad \forall
f\in\tilde\Lambda^+_q\\
\label{tT3_q}\tilde T_q^{(3)}f&:=T_q^{(0)} (-\Pi_{\cL P}f-\Pi_{\cL S}f-\Pi_{\cL H}f)+\Pi^-_{\cL S}f;\\
T^{(3)}_{q+1}(g)&:=T^{(0)}_{q+1}(g-\Pi_{d\cL P}g)+\Pi_{\cL P}^-g,\qquad \qquad g\in\Lambda_{q+1}^+;
\label{T3_q+1}
\\ \label{A1_def}
A^{(3)}f&:=dT^{(3)}_qf +T^{(3)}_{q+1}df,\quad F(f):=A^{(3)}f+\Pi_{\cL H}f,\quad \forall
f\in\tilde\Lambda^+_q.
\end{align}
Although the definition of $F$ involves the projection operator $\Pi_{d\cL P}$, the operator $F-I$ has an {\it unexpected} gain of derivative by avoiding the second estimate in \re{e52dP}.  
\begin{prop}\label{value-F} Assume that \rea{r0>} hold. Then
\bpp\item
$F=A^{(0)}$ on $\widetilde{\cL R} ^s$ and $F=A^{(3)}=A^{(1)}$ on $\widetilde{\cL R} ^+\oplus S$. Consequently, $A^{(1)}$, defined on $\cL R^+\oplus S$, is an extension of $F|_{\widetilde{\cL R} ^+}$.
\item $F=A^{(3)}=I$ on $\cL P\oplus\cL S$, and $F=I$ on $\cL H$.
\item  The $F$, defined on $\widetilde{\cL R} ^\e$ extends to an isomorphism $F$ from $\Lambda^\e$ defined by
   $$ 
    F= A^{(1)} \quad \text{on $\cL R^\e$}, \quad F=I \quad \text{on $\cL P\oplus\cL H\oplus \cL S$}.
 $$ 
 In particular, $F$ is an isomorphism on $\cL Z^{s}_q$ and we actually have
  \eq{e2F}
  e_2(I-F)\geq 1/2.
  \eeq
 \item We have 
 \eq{R4}
f=dP_q^{(0)}f+P_{q+1}^{(0)}df+H_qf
\eeq
for \eq{}\label{P0P0}
P_q^{(0)}:= T_q^{(3)}(F^{-1})^2d T_q^{(3)}, \quad P_{q+1}^{(0)}=F^{-1}T_{q+1}^{(3)}.
\eeq
\item In particular, $e_{2}(I-F)\geq 1/2$ if \rea{ss'}-\rea{ss'-M} and $r_0>1/2$ hold.
\epp
\end{prop}
\begin{proof}
Let $F^{-1}$ be the inverse of $F$ on $\Lambda_q^s$. With $F=I$ on $ \cL P \oplus S \oplus \cL H$, we will use the estimate for $B^{(1)}$
to derive the estimate for $F^{-1}$ via $F^{-1}|_{\cL R^s}=B^{(1)}$.

 We have obtained
$$
f=F^{-1}dT_q^{(3)}f+F^{-1}T_{q+1}^{(3)}df+\Pi_{\cL H}f, \quad f\in\tilde\Lambda^+_q.
$$
On $\cL R^s\oplus \cL S$ we have $T_q^{(3)}=T_q^{(1)}$ and $F=A^{(1)}$. Let $f\in \widetilde{\cL R}^s\oplus \cL S$.  Then by \re{B1dR} we have
$F^{-1}dT_q^{(3)}f=dT_q^{(2)}f=dT_q^{(1)}(B^{(1)})^2dT_q^{(1)}f$. 
The latter equals $dP_q^0f$ since $dT_q^{(1)}f\in\cL Z^s_q\ominus\cL H\subset\cL R^s \oplus \cL S$  
 by  \re{dTq}.
Therefore,
$$ 
F^{-1}d T_q^{(3)}f =d P_q^{(0)}f, 
$$ 
and thus \re{R4} hold on $\widetilde{\cL R} ^s\oplus\cL S$. 
 On $\cL H$, we have $P_q^{(0)}=T_q^{(3)}=0$,  $T_{q+1}^{(3)}d=0$. Thus $F=I$ and \re{R4} holds on $\cL H$. On $\cL P$, we have $T_q^{(3)}=0$ and $T_{q+1}^{(3)}d=I$, so $F=I$.  and \re{R4} holds on $\cL P$.

As mentioned in \rrem{condP}, condition \rea{r0>} is ensured if $ b M\in\Lambda^{3+\e}$. Thus, Assertion $(e)$ holds.
 \end{proof}
Therefore, we have obtained the following estimates for the decomposition \re{R4}.
\begin{prop}\label{prop:Poly} Assume that \rea{ss'}-\rea{ss'-M} and $r_0>1/2$ hold. The homotopy formula \rea{R4} holds on $\tilde \Lambda^+_q$. Furthermore,
\ga{} \label{e2pq}
|P_q^{(0)}f|_r \leq C_{r,\e''}(\|\rho\|_2)
|\rho|^*_{5/2+\e''}(|\rho|^*_{r+3} |\rho|_{2+\e''}^{*\tau}| f|_{\e''}  +|f|_{r}),\quad r>1/2;\\
  e_{3}(P_{q+1}^{(0)}) \geq 1/2,\quad  e_2( H_q) =\infty.
  \label{oldPq+1}
\end{gather}
\end{prop}
\begin{proof}
We have $e_{2}(T_q^{(3)}(F^{-1})^2)\geq 1/2$. Then for $r>0$
\aln{}
&|P_q^{(0)}f|_{r+1/2}\leq C_{r,\e''}
|\rho|_{2+\e''}^{*\tau}|\rho|^*_{r+5/2}|dT_q^{(3)}f|_{\e''}+|dT_q^{(3)}f|_{r}\\
&\qquad\leq C_{r,\e''}'
(|\rho|_{2+\e''}^{*\tau}|\rho|^*_{r+5/2}(|\rho|^*_{3+\e''}|f|_{\e''}+|f|_{1/2+\e''})+|\rho|^*_{r+3}|f|_{\e''}+|f|_{r+1/2}).
\end{align*}
We have $\|\rho\|^*_{r+5/2}\|\rho\|^*_{3+\e''}\leq C_{r,\e''} \|\rho\|^*_{r+3-\e''}\|\rho\|^*_{5/2+2\e''}$ and
$$
|\rho|^*_{r+5/2}|f|_{1/2+\e''}\leq C_{r,\e''}(\|\rho\|^*_{r+3-\e''}\|f\|_{2\e''}+|\rho|^*_{5/2+2\e''}\|f\|_{r+1/2-\e''}). 
$$
Replacing $r+1/2>1/2$ by $r>1/2$ we get \re{e2pq}.

By \re{e52dP}, we have $e_{3}(\Pi_{d\cL P})=\infty$. Since $e_{2}(F^{-1}-1)\geq1/2$ and $e_{2}(T^{(0)}_{q+1})\geq1/2$, \rl{e1e2} implies that $e_{3}(P_{q+1}^{(0)})\geq1/2$.
\end{proof}

\begin{rem}It is clear that on $d\Lambda^{\e}_q$, \re{beta_j} implies that  $\Pi_{\cL P}^-$ does not depend on $\beta_j^{(\e)}$. In fact, $\all_i^{(\e)},\beta_j^{(\e)}, \gamma_\ell^{(\e)}$ do not depend on $\e$ in the sense that
$$
(\all_i^{(\e)},\gamma_\ell^{\e})=(\all_i^{(\e')},\gamma_\ell^{(\e')})|_{\Lambda_q^{\e}}, \quad \beta_j^{(\e)}= \beta_j^{(\e')}|_{d\Lambda_q^\e},
\quad 0<\e'<\e.
$$
Consequently, we can also drop the superscript $(\e)$ in $\all_i^{(\e)},\beta_{j}^{(\e)},\gamma_\ell^{(\e)}$. Then   we still have  estimates
$$
|\all_{i}(f)|\leq C_\e|f|_\e, \quad |\beta_{j}(df)|\leq C_\e|df|_\e,
\quad |\gamma_{\ell}(f)|\leq C_\e|f|_\e
$$
for $f\in\Lambda_q^{\e}$, where $C_\e<\infty$ is independent of $f$ and $df$, respectively.
\end{rem}

With Propositions \ref{value-F} and \ref{prop:Poly} we have adapted Polyakov's construction to our problem. Notice that $P_{q+1}^{(0)}$ gains $1/2$ derivative. However, $P_{q}^{(0)}$ has
no gain in derivatives. In next section, we will find different homotopy operators $P_{q}^{(1)}$ that gains $1/2$ derivative, while
$P_{q+1}^{(1)}$ has no gain in derivatives. This is already useful since $P_{q}^{(1)}$ is a $\db$ solution operator that gains $1/2$ derivatives. Furthermore, we will use the full $a_q$ condition to construct a final version of homotopy operators $P_q$ and $P_{q+1}$ with both gaining $1/2$ derivative.
 
\setcounter{thm}{0}\setcounter{equation}{0}
\section{An improved Polyakov homotopy formula} \label{sect:improve}

To construct a homotopy formula with homotopy operators that have better regularity, we will use the original approximate homotopy formula again. The  scheme is the following: When comparing the new homotopy operators with the original ones, we notice that the difference operators gain derivatives. Therefore, we will use the algebraic properties of the original homotopy operators to modify  a first version of homotopy operators. The ultimate homotopy operators will have the desired regularity.

\medskip

\subsection{Improved $P_q$}   Recall that $A^{(1)}$ is an isomorphism on $\cL R^+\oplus \cL S$ and it has an inverse $B^{(1)}$.
Also $A^{(1)}=A^{(0)}$ on $\cL R^+$. Therefore, $A^{(0)}$ is an isomorphism on $\cL R^+$. Let $B^{(0)}$ be its inverse.
Set
$$
B^{(0)}=I-\hat B^{(0)}, \quad \text{on $\cL R^+$.}
$$
Recall that
$F$ is an isomorphism from $\cL R^ s \oplus\cL S\oplus\cL P\oplus\cL H$ onto itself.
Set
$$
F=I-\hat F, \quad G:= F^{-1}=I-\hat G, \quad G^2=I-\hat G_2, \quad \text{on $\Lambda_q^+$}
$$
Note that
$$ 
\hat G_2=2\hat G-\hat G^2.
$$ 

We now use $dT_q^{(0)}=A^{(0)}-T_{q+1}^{(0)}d$ to modify the homotopy operators. Recall that $ T_q^{(3)}=T^{(0)}_q+ \tilde T^{(3)}_q$ and $P_q^{(0)}=T_q^{(3)}G^2dT_q^{(3)}$. Let us  rearrange the terms as
\ga{}\label{P0q_modify}
P_q^{(0)}= T^{(0)}_qdT_q^{(0)}+T^{(0)}_q\hat G_2T_{q+1}^{(0)}d-\tilde T_q^{(3)}G^2T_{q+1}^{(0)}d 
+\tilde P_q^{(0)}
\end{gather}
with
\al{}
\label{Pti0_exp} \ti P_q^{(0)}&:= \tilde T_q^{(3)}G^2(A^{(0)}+d\tilde T_q^{(3)}) 
-T_q^{(0)}\hat G_2 (A^{(0)}+d\tilde T_q^{(3)}) 
+
T^{(0)}_q d \ti T_q^{(3)}.
\end{align} 
Recall that $P_{q+1}^{(0)}=GT_{q+1}^{(3)}$. We now transport the second and third terms in \re{P0q_modify} and the second term in $T_q^{(0)}dT_q^{(0)}=T_q^{(0)}A^{(0)}-T_q^{(0)}T_{q+1}^{(0)}d$ to $P_{q+1}^{(0)}$.
Therefore,
\al{}\label{Pr1}
    dP_q^{(0)}&+P_{q+1}^{(0)}d
 =dP^{(1)}_q+P_{q+1}^{(1)}d\quad \text{on $\tilde {\Lambda}^s_q$}
 \end{align}
 with
 \al{}
 P^{(1)}_q &:=T_q^{(0)}A^{(0)} +{\tilde P}_q^{(0)},\\
\label{Pq+1}
P^{(1)}_{q+1}&:=d\bl -T_q^{(0)}T_{q+1}^{(0)}
\br+P_{q+1}^{(0)}+\tilde P_{q+1}^{(0)},\\
\tilde P_{q+1}^{(0)}&:=
d( T^{(0)}_q\hat G_2T_{q+1}^{(0)}-\tilde T_q^{(3)}G^2T_{q+1}^{(0)}).  
\label{tpq+1}
\end{align}

To estimate $P_q^{(1)}$, mainly $\tilde P_q^{(0)}$, we must deal with the differential  $d$ that appears in the middle of composition \re{Pti0_exp}. Thus in additional to notation $e_b(T)$, we introduce the following notation:
\begin{defn}We say that $T$  gains {\it exact} $\e$ derivative and write $\tilde e_{b}(T)= \e$, if   
\ga{}\label{defte}
|T f|_{r+\e}\leq C_{b,r,\e''}(\|\rho\|_2) (|\rho|^{*\tau}_{2+\e''}
|\rho|^*_{b+r+\e}|f|_{\e''}+ |f|_{r}), \quad \forall \, r>0.
\end{gather}
Finally, in a slightly weaker form,  we say $\hat e_b(T)=1/2$ if
\ga{}
|T f|_{r+1/2}\leq C_{b,r,\e''}(\|\rho\|_2)(|\rho|^{*\tau}_{2+\e'' }
|\rho|^*_{b+r+1/2}|\rho|^*_{5/2+\e''}|f|_{\e''}+ |f|_{r}), \quad \forall \, r>0.
\end{gather}
\end{defn}

It is obvious that $e_{b}(T)\geq\e$ implies $\tilde e_{b}(T)=\e$.  
\le{tilde-e}
Let  $T_1, T_2$ be linear maps from $\Lambda^+_{q}(\Om,V)$ to $ \Lambda^+_{q'}(\Om,V)$.
Assume that $\tilde e_2(T_1)=1/2$. If
$e_{3}(T_2)\geq1$, then $\hat e_{5/2}(T_1T_2)=1/2$. 
\ele 
\begin{proof} For $r>0$, we have 
\aln{}|T_1T_2f|_{r+1/2}&\leq C_{r,\e''}(|\rho|^{*\tau}_{2+\e''}  |\rho|_{r+5/2}|T_2f|_{\e''}+|T_2f|_{r})\\
&\leq C_{r,\e''}(|\rho|^{*2\tau}_{2+\e''}   |\rho|_{r+5/2}|\rho|_{3+\e''}|f|_{\e''}+|\rho|^{*\tau}_{2+\e''} |\rho|_{r+3}|f|_{\e''}+ |f|_{r}).
\end{align*} 
Then the estimate follows from
\eq{rs52}
|\rho|^*_{r+5/2}|\rho|^*_{3+\e''}\leq C_{r,\e''}|\rho|^*_{r+3-\e''}|\rho|^*_{5/2+2\e''}. \qedhere
\eeq 
\end{proof} 
We now prove the following 
which also improves \re{oldPq+1}.
\le{PQ1} Suppose that \rea{ss'}-\rea{ss'-M} and $r_0>1/2$ hold.  On $\tilde\Lambda_q^+$, we have $I=dP_q^{(1)}+P_{q+1}^{(1)}d+ H_q$. We also have
\ga{}
\label{pq72} 
 \tilde e_{5/2}(\tilde P_q^{(0)})= 1/2,\quad 
\tilde e_{5/2}(P_q^{(1)})=1/2,
\quad   e_2(\cL H_q)=\infty, 
  \\
 \hat e_{5/2}(\tilde P_{q+1}^{(0)})=1/2,\quad 
  \label{hat52} 
   \hat e_{5/2}(P_{q+1}^{(0)})=1/2, 
   \\
\label{e3pq1}     e_3(d( T_q^{(0)}T_{q+1}^{(0)}))\geq0.
\end{gather}
\end{lemma}
\begin{proof}  
To verify \re{pq72}, we recall  
$$
\ti P_q^{(0)}=\tilde T_q^{(3)}G^2(A^{(0)}+d\tilde T_q^{(3)})-T_q^{(0)}\hat G_2 (A^{(0)}+d\tilde T_q^{(3)})+
T^{(0)}_q d \ti T_q^{(3)}.
$$
Recall from \re{tT3_q} that $\tilde T_q^{(3)}=\Pi_{\cL S}^--T_{q}^{(0)}(\Pi_{\cL S}+\Pi_{\cL H}+\Pi_{\cL P})$.
Since $e_2(T_q^{(0)})\geq1/2$ and $e_2(\Pi^-_{\cL S})=e_2(\Pi_{\cL H})=e_2(\Pi_{\cL P})=\infty$, 
by \rl{e1e2} we have $e_2(\tilde T_q^{(3)})=\infty$. By \re{e2F}, we have $e_2(\hat F)\geq1/2$.
Then by \rp{hatB-est}, we get $e_2(\hat G)\geq1/2$ and $e_2(\hat G_2)\geq1/2$. Now it is easy to see that  $e_2(\tilde T_q^{(3)}G^2A^{(0)})=\infty$ and $e_2(T_q^{(0)}\hat G_2 A^{(0)})\geq1/2$.

We have $e_2(T_q^{(0)})\geq2$ and
 $e_3(d\tilde T_q^{(3)})=e_2(\tilde T_q^{(3)})=\infty$. Hence $e_3(d\tilde T_q^{(3)})\geq1$, and by \rl{tilde-e}  we get 
\eq{52hatT}
\hat e_{5/2}(T_q^{(0)}d\tilde T_q^{(3)})=1/2
\eeq
and hence
$\hat e_{5/2}(\tilde P_q^{(0)})=1/2$. 
 We have verified \re{pq72}.
 
We now estimate
 $$
\tilde P_{q+1}^{(0)}=d( T^{(0)}_q\hat G_2T_{q+1}^{(0)}-\tilde T_q^{(3)}G^2T_{q+1}^{(0)})=A^{(0)} \hat G_2T_{q+1}^{(0)}-T_{q+1}^{(0)} d\hat G_2T_{q+1}^{(0)}-
d(\tilde T_q^{(3)}G^2T_{q+1}^{(0)}).
$$
We have $e_2(T_{q+1}^{(0)})\geq1/2$ and $e_2(\hat G_2T_{q+1}^{(0)})\geq1$. Hence, $\hat e_{5/2}(T_{q+1}^{(0)} d\hat G_2T_{q+1}^{(0)})=1/2$ by \rl{tilde-e}. To estimate  $\hat e_{5/2}(d\tilde T_q^{(3)}G^2T_{q+1}^{(0)})$, we recall from  \re{tT3_q} that $$
d\tilde T_q^{(3)}=-dT_{q}^{(0)}(\Pi_{\cL S}+\Pi_{\cL H}+\Pi_{\cL P})+\Pi_{\cL S}=T_{q+1}^{(0)}(d\Pi_{\cL P})-
A^{(0)}(\Pi_{\cL S}+\Pi_{\cL H}+\Pi_{\cL P})+\Pi_{\cL S}.
$$
Therefore, we have
$$
d(\tilde T_q^{(3)}G^2T_{q+1}^{(0)})= \bl \Pi_{\cL S}-A^{(0)}(\Pi_{\cL S}+\Pi_{\cL H}+\Pi_{\cL P}) \br G^2T_{q+1}^{(0)}+T_{q+1}^{(0)}(d\Pi_{\cL P})G^2T_{q+1}^{(0)}.
$$
We have $e_3((d\Pi_{\cL P})G^2T_{q+1}^{(0)})=e_2(\Pi_{\cL P}G^2T_{q+1}^{(0)})\geq1$. We get $\hat e_{5/2}(T_{q+1}^{(0)}(d\Pi_{\cL P})G^2T_{q+1}^{(0)})=1/2$ by \rl{tilde-e}. Now it is easy to see that $\hat e_{5/2}(d(\tilde T_q^{(3)}G^2T_{q+1}^{(0)}))=1/2$ and hence $\hat e_{5/2}( \tilde P_{q+1}^{(0)})=1/2$. By \re{P0P0} and \re{T3_q+1}, we have
$$ 
P_{q+1}^{(0)}=F^{-1}T_{q+1}^{(3)}=F^{-1}(T^{(0)}_{q+1}(I-\Pi_{d\cL P})+\Pi_{\cL P}^-).
$$
Using  $e_3(\Pi_{d\cL P})=\infty$ and \rl{tilde-e}, we get $\hat e_{5/2}(F^{-1}(T^{(0)}_{q+1}(\Pi_{d\cL P})))=1/2$.
Now it is easy to see $\hat e_{5/2}(P_{q+1}^{(0)})=1/2$. We have verified \re{hat52}.

Note that \re{e3pq1} follows from $e_3(d T_q^{(0)}T_{q+1}^{(0)})\geq e_2(d T_q^{(0)}T_{q+1}^{(0)})-1\geq0$.  We cannot improve \re{e3pq1}. In the next two subsections we will transform $d( T_q^{(0)}T_{q+1}^{(0)})$ to get better homotopy operators for the $a_q$ domains. 
\end{proof}

\subsection{Improved $P_q$ and $P_{q+1}$  on convex side of the boundary}

Suppose $bM$ has $n-(q+1)$ positive eigenvalues for $1 \leq q+1 \leq n-1$. Then we have an operator $A^{(0)}_{q+1}=dT^{(0)}_{q+1}+T^{(0)}_{q+2}d$ for the higher degree. In this case, we can improve $P^{(1)}_{q+1}$ in $ d P^{(1)}_q f +P^{(1)}_{q+1} df$. We write \eq{hidden}
dT_q^{(0)}T_{q+1}^{(0)}=A^{(0)}T_{q+1}^{(0)}-T_{q+1}^{(0)}dT_{q+1}^{(0)}=A^{(0)}T_{q+1}^{(0)}
-T_{q+1}^{(0)}A_{q+1}^{(0)}
+ T^{(0)}_{q+1} T_{q+2}^{(0)}d.
\eeq
Thus using $d^2=0$, we obtain
\al{}\label{dP1f} 
 dP^{(1)}_qf+P^{(1)}_{q+1}df&=dP_q^+f+P_{q+1}^{+}df,\quad \forall f\in \tilde\Lambda^+_q, 
 \end{align}
 where
 \al{}\label{Pr1+}
 P_q^{+}&:=P^{(1)}_q,\quad 
 P_{q+1}^{+}=  P_{q+1}^{(0)}+ \tilde P_{q+1}^{(0)}
 -A^{(0)}T_{q+1}^{(0)}+T_{q+1}^{(0)}A_{q+1}^{(0)}.
\end{align}
\begin{rem}\label{density} Strictly speaking, we derive \re{dP1f} via \re{hidden}, where the latter requires $f\in\tilde\Lambda_q^a$ with $a>1/2$.  Then the general case of \re{dP1f} follows from approximating $f\in\tilde\Lambda_q^a$ by smooth forms $f_j$ on $\ov{M}$ such that $f_j-f,df_j-df$ tend $0$ in $\Lambda^\e(M)$ norms (with $0<\e<a$). 
 This approximation holds for the $\db$ operator. 
\end{rem}
 \le{PQ1+} Assume that $M$ and $(n-q)$ convex domains. On $\tilde\Lambda_q^+$, we have $I=dP_q^{+}+P_{q+1}^{+}d+ H_q$. We have
\eq{PP+}
\tilde e_{5/2}(P_q^{+})=1/2,\qquad 
\hat e_{5/2}(P_{q+1}^{+})=1/2,
\qquad e_2(H_q)=\infty.
\eeq
\end{lemma}

\subsection{Improved $P_q$ and $P_{q+1}$  on concave side of the boundary}

Suppose that $bM$ has $q=(q-1)+1$ negative eigenvalues for $0 \leq q-1 \leq n-2$. Then we have an operator $A^{(0)}_{q-1}=d T^{(0)}_{q-1}+T^{(0)}_{q}d$ for the lower degree.
In this case we can improve $P^{(0)}_q$ in the expression $ d P^{(0)}_q f + P^{(0)}_{q+1} df$.

Using $d^2=0$ we write
 $dT_q^{(0)}d = d (A^{(0)}_{q-1} - d T^{(0)}_{q-1}) = dA^{(0)}_{q-1}$. Since $P_q^{(0)} = T^{(0)} d T^{(0)}_q + \ti P_q^{(0)}$, we have
\al{}\label{dPq0}
dP_q^{(0)}f+P_{q+1}^{(0)} df 
&=dP_q^{-}f+P_{q+1}^{-} df,\quad \forall f\in\tilde{\cL R}^+_q
 \end{align}
 where
 \al{}\label{Pr1-}
 P_q^{-}&:= \tilde P_q^{(0)}+A_{q-1}^{(0)}T_q^{(0)},\quad 
 P_{q+1}^{-}:= P^{(0)}_{q+1}+\tilde P_{q+1}^{(0)}.
\end{align}
As  \rrem{density}, strictly speaking \re{dPq0} is derived for $f\in\tilde\Lambda^a_q$ with $a>1/2$.
Then the general case of \re{dPq0} is obtained via the simultaneous smooth approximation of $f,df$. The approximation is satisfied for the $\db$ operator.   Therefore, we have obtained the following.

\le{PQ1-}Suppose that \rea{ss'}-\rea{ss'-M} and $r_0>1/2$ hold. Assume that $M$ and $(q+1)$ concave domain with $q>0$. On $\tilde\Lambda_q^+$, we have $I=dP_q^{-}+P_{q+1}^{-}d+H_q$ and
\eq{pp-}
\tilde e_{5/2}(P_{q}^-)=1/2, \quad  
\hat e_{5/2}(P_{q+1}^-)= 1/2, \quad e_2(H_q)=\infty. 
 \eeq
 \ele

 \subsection{Improved $P_q$ and $P_{q+1}$  for  $a_q$ domains}
 
When the boundary of $M$ has both concave and convex points, we can still modify the homotopy operators.

Let us recall \re{T3_q}-\re{A1_def}:
\al{}\label{T3_q-5}
T^{(3)}_qf&:= T_q^{(0)}f+\tilde T_q^{(3)}f,\quad \forall
f\in\tilde\Lambda^+_q;\\
\label{tT3_q-5}\tilde T_q^{(3)}f&:=T_q^{(0)} (-\Pi_{ P}f-\Pi_{\cL S}f-\Pi_{\cL H}f)+\Pi^-_{\cL S}f;\\
T^{(3)}_{q+1}(g)&:=T^{(0)}_{q+1}(g-\Pi_{d\cL P}g)+\Pi_{\cL P}^-g,\hspace{18ex} \forall g\in\Lambda_{q+1}^+;
\label{T3_q+1-5}
\\ \label{A1_def-5}
A^{(3)}f&:=dT^{(3)}_qf +T^{(3)}_{q+1}df,\quad F(f):=A^{(3)}f+\Pi_{\cL H}f,\quad \forall
f\in\tilde\Lambda^+_q.
\end{align}
We also recall \re{R4}-\re{P0P0}: 
\ga{}\label{R4-5}
f=dP_q^{(0)}f+P_{q+1}^{(0)}df+\Pi_{\cL H}f,\\
\label{P0P0-5}
P_q^{(0)}= T_q^{(3)}(F^{-1})^2d T_q^{(3)}, \quad P_{q+1}^{(0)}=F^{-1}T_{q+1}^{(3)}.
\end{gather}

Write
$$ 
P_q^{(0)}=T^{(0)}_qdT_q^{(0)}-\tilde T_q^{(3)}G^2T_{q+1}^{(0)}d+T_q^{(0)}\hat G_2T_{q+1}^{(0)}d+\tilde P_q^{(0)}
x$$ 
with
\al{}
\label{Pti0_exp-last} 
\ti P_q^{(0)}&=\tilde T_q^{(3)}G^2(A^{(0)}+d\tilde T_q^{(3)}) 
-T_q^{(0)}\hat G_2(A^{(0)}+d\tilde T_q^{(3)})c
+
T^{(0)}_q d \ti T_q^{(3)}. 
\end{align}
Then we have 
\begin{align}\label{dPq}
 d P^{(1)}_q &= d \bl T^{(0)}_q A^{(0)}-\tilde T_q^{(3)}G^2T_{q+1}^{(0)}d+T_q^{(0)}\hat G_2T_{q+1}^{(0)}d  + \ti P^{(0)}_q \br
  \\&= d \bl T^{(0)}_q d T^{(0)}_q + T^{(0)}_q T^{(0)}_{q+1} d -\tilde T_q^{(3)}G^2T_{q+1}^{(0)}d+T_q^{(0)}\hat G_2T_{q+1}^{(0)}d +  \ti P^{(0)}_q\br.\nonumber
\end{align}
Note that the operators $T_q^{(0)}, T_{q+1}^{(0)}$ are constructed via the same partition of unity $\{\chi_\ell\}$ of $\ov M$, given by  \re{T0_exp}:
\eq{T0_exp-5}
  T^{(0)}_{\tilde q} = \sum_k \chi_k T^+_{\tilde q,U_k} + \sum_\ell \chi_\ell T^-_{\tilde q,U_\ell},\quad\tilde q=q,q+1,
\eeq
where we split the operators into those for  $T^-_{i,U_\ell}$ with $U_\ell\cap b^-_{q+1}\neq\emptyset$ and 
 $T^+_{i,U_k}$ for the rest,  where the latter also includes
  $T_{m,U_m}$ that  gains  a full derivative on H\"older-Zygmund spaces in the interior of $U_m$ when $U_m\cap bM=\emptyset$. Here we have used a result in~\ci{MR999729}.

The presence of cut-off functions $\chi_k$ in \re{T0_exp-5} guarantees that each term in the summand \re{T0_exp-5} is globally defined on $\ov M$.  Therefore, we can transport individual term maintaining new formulas to be globally defined, after they are modified from the ones in previous sections. 

We start with
\begin{gather*}
 d T^{(0)}_q d T^{(0)}_q  = d \Bl \sum_k \chi_k T^+_{q,U_k}   d T^{(0)}_q \Br + d \Bl \sum_\ell \chi_\ell T^-_{q,U_\ell} dT^{(0)}_q \Br,
 \\
d T^{(0)}_q T^{(0)}_{q+1} d = d \Bl \sum_k \chi_k T^+_{q,U_k}  T^{(0)}_{q+1} d \Br  + d \Bl  \sum_\ell \chi_\ell T^-_{q,U_\ell}T^{(0)}_{q+1} d \Br .
\end{gather*}
Since $A^{(0)} = d T^{(0)}_q + T^{(0)}_{q+1} d$, we have
\al{} \label{dTdT+dTTd}
& d T^{(0)}_q d T^{(0)}_q + d T^{(0)}_q T^{(0)}_{q+1} d
    \\ &\quad = d \Bl \sum_k \chi_k T^+_{q,U_k} A^{(0)} \Br
+ d \Bl \sum_\ell \chi_\ell T^-_{q,U_\ell} dT^{(0)}_q \Br + d \Bl  \sum_\ell \chi_\ell T^-_{q,U_\ell}T^{(0)}_{q+1} d \Br.
\nonumber
\end{align} 
By \re{Pq+1}, we get
\begin{equation} \label{Pq+1_d}
 P^{(1)}_{q+1} d = \bl P_{q+1}^{(0)} +\tilde P^{(0)}_{q+1} \br d 
 -dT_q^{(0)}T_{q+1}^{(0)} d .
\end{equation}
where $P_{q+1}^{(0)}$ and $\tilde P^{(0)}_{q+1}$ are defined by \re{P0P0} and \re{tpq+1}, respectively.
We decompose the term $dT_q^{(0)}T_{q+1}^{(0)} d$ as
\begin{align} \label{-dTTd} 
  dT_q^{(0)}T_{q+1}^{(0)} d &=  d \Bl  \sum_k \bl \chi_k T^+_{q,U_k} \br T^{(0)}_{q+1} d \Br + d \Bl  \sum_\ell \bl \chi_\ell T^-_{q,U_\ell} \br T^{(0)}_{q+1} d \Br
  \\&=  d \Bl \sum_{k,m} \bl \chi_k T^+_{q,U_k} \br \chi_m
 T^+_{q+1,U_m} d \Br + d \Bl \sum_\ell \bl \chi_\ell T^-_{q,U_\ell} \br T^{(0)}_{q+1} d \Br
 \nonumber\\
& \quad  +\Bl d\sum_{k,m} \bl \chi_k T^+_{q,U_k} \br \chi_m
 T^-_{q+1,U_m}  \Br d.
 \nonumber
\end{align}
Let us first treat the terms in the last line. When $U_k$ intersects $b^+_{n-q}( b M)$ and $U_m$ intersects $b^-_{q+1}M$, we have $U_k \cap U_m  = \emptyset$. Since  $\supp\chi_\ell\Subset U_\ell$, we may assume that for the non-zero terms in the last line of \re{-dTTd}, we have $U_k\Subset M$ and hence this   $\chi_k T^+_{q,U_k}$ gains a full derivative.  Then
\begin{align*}
   |d (\chi_k T^+_{q,U_k} \chi_m
 T^-_{q+1,U_m})g|_{r+\yh} 
   &\leq |(\chi_k T^+_{q,U_k} \chi_m
 T^-_{q+1,U_m})g|_{r+\frac32} 
  \\
   &\leq C_r| T^-_{q+1,U_m} g|_{r+\yh} \leq C_{r,\e}(|\rho|_{2}) (|\rho|^*_{r+2+\yh} |f|_\e + |f|_r).  
\end{align*}
This  shows that
\ga 
 e_2(P_{q+1}^*)\geq1/2, \quad P_{q+1}^*:= d\sum_{k,m}  \chi_k T^+_{q,U_k}  \chi_m
 T^-_{q+1,U_m}.
 \label{e3Ps}
\end{gather}

By combining \re{dPq}-\re{-dTTd} we get
\begin{align*}
  d P^{(1)}_q +(P^{(1)}_{q+1}&+P^{*}_{q+1}) d = d \Bl \sum_k \chi_k T^+_{q,U_k} A^{(0)} + \ti P^{(0)}_q \Br + \bl P_{q+1}^{(0)} +\tilde P^{(0)}_{q+1} \br d
  \\ &\quad + d \sum_\ell \chi_\ell T^-_{q,U_\ell} dT^{(0)}_q - d \Bl \sum_{k,m} \bl \chi_k T^+_{q,U_k} \br \bl \chi_m T^+_{q+1,U_m} \br d \Br.
\end{align*}
We now use the convexity/concavity to modify the last two terms in the above line.
Since $T_{\tilde q, U_\ell}^-$ satisfies $U_\ell\cap b^-_{p+1}\neq\emptyset$ and $q\geq1$, we have
$T^-_{q,U_\ell} d = A^-_{q-1,U_\ell} - d T^-_{q-1, U_\ell}$ and thus
\begin{align*}
 d \sum_\ell \chi_\ell T^-_{q,U_\ell}& dT^{(0)}_q
= d \sum_\ell \chi_\ell (A^-_{q-1,U_\ell} - d T^-_{q-1, U_\ell}) T^{(0)}_q
\\ \qquad &= d \sum_\ell \chi_\ell A^-_{q-1,U_\ell} T^{(0)}_q - d \left[ d \bl \chi_\ell T^-_{q-1, U_\ell} T^{(0)}_q \br - (d\chi_\ell)
\wedge T^-_{q-1, U_\ell} T^{(0)}_q \right]
\\ \qquad &= d \Bl \sum_\ell \chi_\ell A^-_{q-1,U_\ell} T^{(0)}_q
+ (d \chi_\ell)\wedge T^-_{q-1, U_\ell} T^{(0)}_q \Br.
\end{align*}
Next, we modify each term in
\[
  d \Bl \sum_{k,m} \bl \chi_k T^+_{q,U_k} \br \bl \chi_m T^+_{q+1,U_m} \br d \Br.
\]
We have operator $A^+_{q+1,U_m}=dT^+_{q+1,U_m}+T^+_{q+2,U_m}d$, where both $T^+_{q+2,U_m}, T^+_{q+1, U_m}$ gain at least $1/2$ derivative. Then we have
\aln{}
  d \Bl \bl \chi_k T^+_{q,U_k} \br \bl \chi_m T^+_{q+1,U_m} \br d \Br
  &=   \chi_k dT^+_{q,U_k}\chi_m T^+_{q+1,U_m}d+E_{1;km}
\\
&= - \chi_k T^+_{q+1,U_k}d(\chi_m T^+_{q+1,U_m})d+E_{2;km}+E_{1;km},\\
&=- \chi_k T^+_{q+1,U_k}\chi_m dT^+_{q+1,U_m}d+E_{3;km}+E_{2;km}+E_{1;km},\\
&=E_{4;km}+E_{3;km}+E_{2;km}+E_{1;km}
\end{align*}
\begin{gather*}
    E_{1;km}:=(d\chi_k)\wedge T^+_{q,U_k}\chi_m T^+_{q+1,U_m}d, \quad
E_{2;km}:=\chi_k A^+_{q,U_k}\chi_m T^+_{q+1,U_m}d,\\
E_{3;km} :=\chi_k T^+_{q+1,U_k}((d\chi_m)\wedge T^+_{q+1,U_m})d, \quad
E_{4;km}:= - \chi_k T^+_{q+1,U_k}\chi_mA^+_{q+1,U_m}d.
\end{gather*}
Therefore, we can write each term in $d(\chi_k T_{q,U_\ell}\chi_mT_{q+1,U_m})d$ as $\tilde E_{j;km}d$ where $\tilde E_{j;km}$ gains $1/2$ derivative; more precisely the term vanishes when $U_k\cap U_m=\emptyset$. When the intersection is non-empty, we have
\eq{e2tE}
e_2(\tilde E_{j;km})\geq1/2.
\eeq

To summarize, we have derived the following identities:
\begin{gather} \nn
  d P^{(0)}_q + P^{(0)}_{q+1} d = d P_q + P_{q+1} d ;
\\ \label{Pq_exp} P_q := \sum_k \chi_k T^+_{q,U_k} A^{(0)} + \ti P^{(0)}_q
+ \sum_\ell \chi_\ell A^-_{q-1,U_\ell} T^{(0)}_q
+ d \chi_\ell\wedge T^-_{q-1, U_\ell} T^{(0)}_q ;
  \\ \label{Pq+1_exp}  P_{q+1} := P_{q+1}^{(0)} +\tilde P^{(0)}_{q+1}
  - \sum_{j,k,m} \ti E_{j;km}-d\sum_{k,m}  \chi_k T^+_{q,U_k}  \chi_m
 T^-_{q+1,U_m}.
\end{gather} 
Here $\tilde P_q^{(0)}$ is  defined by \re{Pti0_exp-last}, and $P_{q+1}^{(0)}, \tilde P_{q+1}^{(0)}$
are defined by  \re{P0P0},  \re{T3_q+1}, and \re{tpq+1}. Then \re{R4} is transformed into 
$$
f=d P_q f+ P_{q+1} d f+H_qf,
$$
 where  $H_q$ is the projection $\Pi_{\cL H}$ from $\Lambda_{q}^{\e''}$ onto $\cL H_q$  with $\cL H_q=\cL H$
being defined by \re{defSH}  in \rl{lemma:KSH}.

 We now state a homotopy formula with both $P_q$ and $P_{q+1}$ gaining $1/2$ derivative.
\begin{thm}\label{defAB+} 
Let $V$ be a holomorphic vector bundle on $\cL M$.
Let
$M$ defined by $ \rho<0$ be a relatively compact $a_q$ domain in $\cL M$ with $0<q\leq n-2$.
Let $P_q, P_{q+1}, H_q$ be defined by \rea{Pq_exp}-\rea{Pq+1_exp}. Let  $r>0$ and $ b M\in\Lambda^{r+3}$. 
For all $f\in \Lambda_q^r(M,V)$ with  $d f\in\Lambda^{r}$, we have
$$ 
f=dP_q f+  P_{q+1}df+H_qf,
$$ 
where $H_qf\in\cL H_q$. Moreover, we have
\al{}
\label{gainP}
|P_{q}g|_{r+1/2}&\leq  C^*_{2,\e}  C_{r,\e}(|\rho|^*_{2+\e})  (|g|_{r}+
|\rho|^*_{r+3}|\rho|^{*\tau}_{2+\e}|g|_{\e}),
\\
|P_{q+1}g|_{r+1/2}&\leq  C_{2,\e}^{*}  C_{r,\e}(|\rho|^*_{2+\e}) ( |g|_{r}+
|\rho|^*_{r+3} |\rho|^{*}_{5/2+\e}
|\rho|^{*\tau}_{2+\e}|g|_{\e}),\\
|H_qf|_{r+1/2}&\leq   C^*_{2,\e}  C_{r,\e,\e}(|\rho|^*_{2+\e})
|\rho|^*_{r+5/2} |\rho|^{*\tau}_{2+\e}|f|_{\e},
\end{align}
where $\e>0$,  and  $C_{2,\e}^*$ 
depends on the norm of the inverse map of $F$ and the projections $\Pi_{\cL S}, \Pi_{\cL H}, \Pi_{\cL P}$ $($resp.~$\Pi_{d\cL P}$ additionally$)$ in $\Lambda^{\e}$ norms. 
\end{thm}
\begin{proof} The estimates are   consequences of Lemmas~\ref{PQ1+} and \ref{PQ1-}, estimates \re{e3Ps}-\re{e2tE},
and \rl{e1e2}  on compositions.
\end{proof}
\begin{rem} The constants $\tau,  C_{r,\e}(|\rho|^*_{2+\e})$ are stable under small $\Lambda^{2+\e}$ perturbations of $ b M$.
  However, we have no claim on the stability of $C_{2, \e}^*$ under small $\Lambda^{j+\e}$ perturbations of $ b M$ shown in \re{B1est} for the inverse of $A^{(0)}$ on $\cL R^\e$. 
   Also, if $b_{n-q}^+M$ is empty, then \re{gainP} also holds for $P_{q+1}$.
\end{rem}

\newcommand{\doi}[1]{\href{http://dx.doi.org/#1}{doi:#1}}
\newcommand{\arxiv}[1]{\href{https://arxiv.org/pdf/#1}{arXiv:#1}}

  \def\MR#1{\relax\ifhmode\unskip\spacefactor3000 \space\fi%
  \href{http://www.ams.org/mathscinet-getitem?mr=#1}{MR#1}}

\nocite{}
\bibliographystyle{alpha}



\end{document}